\documentclass[12pt]{amsart}
\usepackage{amssymb}
\usepackage{txfonts}
\usepackage{eucal}

\usepackage[all]{xy}

\usepackage{tikz}

\usepackage{graphicx}
\usepackage{float}

\usepackage{xspace}

\usepackage[a4paper,body={16.3cm,22.8cm},centering]{geometry}
\usepackage[colorlinks,final,backref=page,hyperindex,pdftex]{hyperref}

\newcommand{\nc}{\newcommand}
\newcommand{\delete}[1]{}

\long\def\ignore#1{}
\nc{\mlabel}[1]{\label{#1}}  
\nc{\mcite}[1]{\cite{#1}}  
\nc{\mref}[1]{\ref{#1}}  
\nc{\mbibitem}[1]{\bibitem{#1}} 

\delete{
\nc{\mlabel}[1]{\label{#1}  
{\hfill \hspace{1cm}{\small\tt{{\ }\hfill(#1)}}}}
\nc{\mcite}[1]{\cite{#1}{\small{\tt{{\ }(#1)}}}}  
\nc{\mref}[1]{\ref{#1}{{\tt{{\ }(#1)}}}}  
\nc{\mbibitem}[1]{\bibitem[\bf #1]{#1}} 
}

\newtheorem{theorem}{Theorem}[section]
\newtheorem{prop}[theorem]{Proposition}

\theoremstyle{definition}
\newtheorem{defn}[theorem]{Definition}
\newtheorem{remark}[theorem]{Remark}
\newtheorem{exam}[theorem]{Example}
\newtheorem{prop-def}{Proposition-Definition}[section]

\newcommand\cal[1]{\mathcal{\MakeUppercase{#1}}}

\newcommand\alphlist{a,b,c,d,e,f,g,h,i,j,k,l,m,n,o,p,q,r,s,t,u,v,w,x,y,z}
\newcommand\Alphlist{A,B,C,D,E,F,G,H,I,J,K,L,M,N,O,P,Q,R,S,T,U,V,W,X,Y,Z}
\newcommand\getcmds[3]{\expandafter\newcommand\csname #2#1\endcsname{#3{#1}}}
\makeatletter
\@for\x:=\alphlist\do{\expandafter\getcmds\expandafter{\x}{cal}{\cal}}      
\@for\x:=\alphlist\do{\expandafter\getcmds\expandafter{\x}{frak}{\mathfrak}}
\@for\x:=\Alphlist\do{\expandafter\getcmds\expandafter{\x}{frak}{\mathfrak}}
\makeatother

\nc{\bfk}{{\bf k}}

\font\cyr=wncyr10

\newfont{\scyr}{wncyr10 scaled 550}
\nc{\sha}{\mbox{\cyr X}}
\nc{\ssha}{\mbox{\bf \scyr X}}

\nc{\Id}{\mathrm{Id}}
\nc{\lbar}[1]{\overline{#1}}
\nc{\ot}{\otimes}
\nc{\dep}{\mathrm{dep}}

\nc{\tred}[1]{\textcolor{red}{#1}} \nc{\tgreen}[1]{\textcolor{green}{#1}}
\nc{\tblue}[1]{\textcolor{blue}{#1}} \nc{\tpurple}[1]{\textcolor{purple}{#1}}

\nc{\li}[1]{\tpurple{\underline{Li: }#1 }}
\nc{\liadd}[1]{\tpurple{#1}}
\nc{\xing}[1]{\tblue{\underline{Xing: }#1 }}
\nc{\yuan}[1]{\tred{\underline{Yuan: }#1 }}
\nc{\markus}[1]{\tred{\underline{Markus: } #1}}
\nc{\dominique}[1]{\tpurple{\underline{Dominique: }#1 }}

\newcommand\scopeclip[1]{\begin{scope}
\clip(-1.1,-0.5)rectangle(1.1,1);#1\end{scope}}
\newcommand\XX[2][]{%
\tikz[line width=0.15ex,x=0.5cm,y=0.5cm,baseline,inner sep=1.5pt,
every node/.style={font=\scriptsize},#1]{
\scopeclip{\draw (135:1.5)--(0,0)--(45:1.5) (0,-0.5)--(0,0);}#2}}
\newcommand\xx[3]{%
\scopeclip{\draw(#1/10,#2/10)--+(#3*45:2.5);}}
\newcommand\xxl[2]{\xx{#1}{#2}3}
\newcommand\xxr[2]{\xx{#1}{#2}1}
\newcommand\xxlr[2]{\xxl{#1}{#2}\xxr{#1}{#2}}
\newcommand\xxh[6]{
\draw(#1/10,#2/10)+(0.5*#3*45+0.5*#4*45:#6) node[above] {$#5$};}
\newcommand\xxhu[4][0.15]{\xxh{#2}{#3}13{#4}{#1}}
\newcommand\stree[1]{\XX{\xxhu[0.25]00{#1}}}
\newcommand\ssstree[2]{\XX[scale=1.3,baseline=0.5ex]{\xx00{1.6}\xx00{2.4}
\xxh00{2.4}3{#1\ \ }{0.8}\xxh001{1.6}{\ \ \,#2}{0.8}
\draw (0,0.9) node {$\ldots$};}}

\nc{\dnx}{\Delta_n A} \nc{\dx}{\Delta A} \nc{\dgp}{{\rm deg_{P}}}
\nc{\dgt}{{\rm deg_{T}}} \nc{\dg}{{\rm deg}} \nc{\ida}{ID($A$)} \nc{\tu}{\tilde{u}} \nc{\tv}{\tilde{v}}
\nc{\nr}{\calr_n} \nc{\nz}{\calz_n} \nc{\fun}{\cala_{n,d}}
 \nc{\fbase}{\calb} \nc{\LF}{\mathrm{RF}} \nc{\FFA}{\mathrm{LF}} \nc{\irr}{\mathrm{Irr}}
 \nc{\result}{\bfk\mathrm{Irr}(S_n)}  \nc{\I}{I_{\mathrm{ID},n}^0}
 \nc{\nrs}{\calr_n^\star} \nc{\ii}{\mathrm{I}} \nc{\iii}{\mathrm{II}}
\nc{\intl}{{\rm int}}\nc{\ws}[1]{{#1}}\nc{\deleted}[1]{\delete{#1}}\nc{\plas}{placements\xspace}

\nc{\bim}[1]{#1}  \nc{\shaop}{\sha_{\Omega}^{+}}  \nc{\shao}{\sha_{\Omega}}
\nc{\bbim}[2]{#1 #2} \nc{\bbbim}[2]{#1,\, #2} \nc{\RBF}{{\rm RBF}}
\nc{\frb}{F_{\RB}} \nc{\shaf}{\ssha_{\tiny{\Omega}}} \nc{\sham}{\diamond_{\tiny{\Omega}}}
\nc{\lf}{\lfloor} \nc{\rf}{\rfloor} \nc{\shan}{\ssha_{\lambda}}
\nc{\rlex}{{\rm {lex}}} \nc{\bb}{\Box} \nc{\ra}{\rightarrow}
\nc{\e}{{\rm {e}}}
\nc{\DDF}{\mathrm{DD}(X,\,\Omega)}\nc{\DTF}{\mathrm{DT}(X,\,\Omega)}
\nc{\bre}{\mathrm{bre}}
\nc{\dec}{\mathrm{dec}}
\nc{\type}{\mathrm{type}}

\def\Ve#1,#2,#3;{\vee_{#1,\,(#2,\,#3)}}
\def\bigv#1;#2;#3;{\bigvee\nolimits_{#1}^{#2;\,#3}}

\begin{document}

\title[Free (tri)dendriform family algebras]{Free (tri)dendriform family algebras}
\author{Yuanyuan Zhang} \address{School of Mathematics and Statistics,
Lanzhou University, Lanzhou, 730000, P. R. China}
\email{zhangyy2017@lzu.edu.cn}
\author{Xing Gao
}
\address{School of Mathematics and Statistics,
Key Laboratory of Applied Mathematics and Complex Systems,
Lanzhou University, Lanzhou, 730000, P.R. China}
\email{gaoxing@lzu.edu.cn}
\author{Dominique Manchon$^{*}$}
\footnotetext{* Corresponding author.}
\address{Laboratoire de Math\'ematiques Blaise Pascal,
CNRS--Unibersit\'e Clermont-Auvergne,
3 place Vasar\'ely, CS 60026,
F63178 Aubi\`ere, France}
\email{Dominique.Manchon@uca.fr}

\date{\today}

\begin{abstract}
In this paper, we first prove that a Rota-Baxter family algebra indexed by a semigroup induces an ordinary Rota-Baxter algebra structure on the tensor product with the semigroup algebra. We show that the same phenomenon arises for dendriform and tridendriform family algebras. Then we construct free dendriform family algebras in terms of typed decorated planar binary trees. Finally, we generalize typed decorated rooted trees to typed valently decorated Schr\"oder trees and use them to construct free tridendriform family algebras.
\end{abstract}

\subjclass[2010]{
16W99, 
16S10, 
13P10, 
08B20, 
}

\keywords{Rota-Baxter algebra, Rota-Baxter family algebra, (tri)dendriform algebra, (tri)dendriform family algebra, typed decorated planar binary trees, typed valently decorated rooted trees.}

\maketitle

\tableofcontents

\setcounter{section}{0}

\allowdisplaybreaks
\section{Introduction}\mlabel{sec:rota}
 The concept of Rota-Baxter family algebra is a generalization of Rota-Baxter algebras~\cite{Guo12}, which was proposed by Guo. It arises naturally in renormalization of quantum field theory~(\cite[Proposition~9.1]{FBP} and \cite[Theorem 3.7.2]{DK} ). The free objects were constructed in~\cite{ZG}, in which the authors described free commutative Rota-Baxter family algebras, and also described free noncommutative Rota-Baxter family algebras via the method of Gr\"{o}bner-Shirshov bases.\\

Dendriform algebras were introduced by Loday~\cite{Lod93} in 1995 with motivation from algebraic K-theory. They have been studied quite extensively with connections to several areas in mathematics and physics, including operads ~\cite{Loday2}, homology~\cite{Fra2}, arithmetics~\cite{Loday3}, Hopf algebras~\mcite{LoRo98,LoRo04,Ron02} and quantum field theory ~\cite{EFMP,Loday1}, in connection with the theory of renormalization of Connes and Kreimer~\cite{CK,CK1,CK2}.\\

Some years later after~\mcite{Lod93}, Loday and Ronco introduced the concept of tridendriform algebra in the study of polytopes and Koszul duality~\mcite{LoRo04}. The construction of free (tri)dendriform algebras can be referred to~\cite{Lod93,LoRo98,LoRo04}. The free objects play a crucial role in the
study of any algebraic structures, such as the construction of free differential algebras in terms of differential monomials,
the construction of free Rota-Baxter algebras~\cite{GK2} which is more involved, the free differential Rota-Baxter algebras~\mcite{GK3}
composing the construction of free differential algebras followed by that of the free Rota-Baxter algebras, and the free integro-differential algebras~\mcite{GGR, GGZ} for analyzing the underlying algebraic structures
of boundary problems for linear ordinary differential equations. In~\cite{ZG}, the authors proposed the concepts of (tri)dendriform family algebras and they obtained that Rota-Baxter family algebras induce (tri)dendriform family algebras. Besides, they construct free commutative (tri)dendriform family algebras.\\

It is natural to ask how to use reduced planar rooted trees (also known as Schr\"oder trees) to construct free (tri)dendriform family algebras. Here we exhibit a way to construct free (tri)dendriform family algebras via typed decorated Schr\"oder trees, whose vertices are decorated by elements of a set $X$ and edges are decorated by elements of a semigroup $\Omega.$ Typed decorated trees are used by Bruned, Hairer and Zambotti in~\mcite{BHZ} to give a systematic description of a canonical renormalisation procedure of stochastic PDEs. They also appear~\cite{RM} in a context of low dimension topology and also appear in~\cite{BLL} for the description of combinatorial species. The concept of typed decorated planar binary trees enables us to construct free dendriform family algebras, and the concept of typed valently decorated Schr\"oder trees will be used to construct free tridendriform family algebras.\\

The layout of the paper is as follows. In Section~\mref{sec:rota}, we prove that any Rota-Baxter family algebra $R$ on a base ring $\bfk$ indexed by a semigroup $\Omega$ induces an ordinary Rota-Baxter algebra structure on $R\otimes \bfk\Omega$ (Theorem~\mref{thm:pp1}). The semigroup $\Omega$ at hand is not necessarily commutative. Then we prove that Rota-Baxter family algebras can induce (tri)dendriform family algebras (Proposition~\mref{prop:RBTD3}). Finally, we give the relationship between (tri)dendriform family algebras and (tri)dendriform algebras (Theorem~\mref{thm:prsut}). Section~\mref{sec:dd} is devoted to typed decorated planar binary trees. We derive a useful recursive expression for typed decorated planar binary trees endowed with the operations $\{\prec_\omega,\succ_\omega\mid\omega\in\Omega\}$ introduced in the third section (Definition \ref{def:freeDF}).
We first construct dendriform family algebras (Proposition~\mref{prop:dend}) in terms of typed decorated planar binary trees, then we prove the freeness of this dendriform family algebra (Theorem~\mref{thm:free1}). In Section~\mref{sec:tri}, closely related to Section~\mref{sec:dd}, we first introduce the concept of typed valently decorated Schr\"oder trees. Secondly,
we construct tridendriform family algebras (Proposition~\mref{prop:trid}) in terms of typed valently decorated Schr\"oder trees, and finally prove the freeness of this tridendriform family algebra (Theorem~\mref{thm:free2}).\\

The related notions of \textsl{matching Rota-Baxter algebra}, \textsl{matching dendriform algebra} and \textsl{matching pre-Lie algebra} are addressed in the recent paper \cite{GGZ2019}. The main difference is that the parameter at hand runs into a set $\Omega$ without any semigroup structure: see Remark 2. 2 (b) therein.\\

\noindent {\bf Notation:} In this paper, we fix a ring $\bfk$ and assume that an algebra is a \bfk-algebra. Denote by $\Omega$ a semigroup, unless otherwise specified.

\section {Rota-Baxter family algebras and (tri)dendriform family algebras}
\mlabel{sec:RBF}
In this section, we first recall the concepts of Rota-Baxter family algebras and (tri)dendriform family algebras. We then give a method to
induce Rota-Baxter (resp.~(tri)dendriform) algebras from Rota-Baxter (resp.~(tri)dendriform) family algebras.

\subsection{Rota-Baxter family algebras}
 Rota-Baxter algebras (first called Baxter algebras) are originated in the work of the American mathematician Glen E. Baxter~\cite{Bax}
in the realm of probability theory.
\begin{defn} Let $\lambda$ be a given element in $\bfk$. A {\bf Rota-Baxter algebra of weight $\lambda$} is a pair $(R, P)$ consisting of an algebra $R$
with a linear operator $P: R \rightarrow R$ that satisfies the Rota-Baxter equation
$$P(a)P(b)=P\bigl(P(a)b  + aP(b)+\lambda ab \bigr),\, \text{ for }\, a, b \in R.  $$
Then $P$ is called a {\bf Rota-Baxter operator of weight $\lambda$}. If further $R$ is commutative, then $(R, P)$ is called a commutative Rota-Baxter algebra of weight $\lambda$.
\end{defn}

The following is the concept of Rota-Baxter family algebra proposed by Guo,
which arises naturally in renormalization of quantum field theory~\cite[Proposition~9.1]{FBP}, see also~\cite{DK}.
\begin{defn}~\rm{(\mcite{FBP,Guo09})}
Let $\Omega$ be a semigroup and $\lambda\in \bfk$ be given.
A {\bf Rota-Baxter family} of weight $\lambda$ on an algebra $R$ is a collection of linear operators $\{ P_{\omega} \mid \omega\in \Omega\}$ on $R$ such that
\begin{equation*}
P_{\alpha}(a)P_{\beta}(b)=P_{\alpha\beta}\bigl( P_{\alpha}(a)b  + a P_{\beta}(b) + \lambda ab \bigr),\, \text{ for }\, a, b \in R\,\text{ and }\, \alpha,\, \beta \in \Omega.
\end{equation*}
Then the pair $(R,\, \{ P_{\omega} \mid \omega \in \Omega\})$ is called a {\bf Rota-Baxter family algebra} of weight $\lambda$.
If further $R$ is commutative, then $(R,\, \{ P_{\omega} \mid \omega\in \Omega\})$ is called a commutative Rota-Baxter family algebra
of weight $\lambda$.
\mlabel{def:pp}
\end{defn}

\begin{defn}
Let $(R,\,\{ P_{\omega} \mid \omega\in \Omega \})$ and $(R', \, \{ P'_{\omega} \mid \omega\in \Omega \})$ be
two Rota-Baxter family algebras of weight $\lambda$.
A map $f: R \rightarrow R'$ is called {\bf a Rota-Baxter family algebra morphism} if
$f$ is an algebra homomorphism and $f\circ P_{\omega}=P'_{\omega}\circ f$ for each $\omega\in \Omega$.
\end{defn}

\noindent The following theorem precises the link between Rota-Baxter family algebras and ordinary Rota-Baxter algebras.

\begin{theorem}
Let $(R,\{P_\omega\mid\omega\in\Omega\})$ be a Rota-Baxter family algebra of weight $\lambda$. Then
$(R \ot\bfk\Omega,P)$ is a Rota-Baxter algebra of weight $\lambda$, where $P: R\ot\bfk\Omega\ra R\ot\bfk\Omega, x\ot \omega\mapsto P_\omega(x)\ot\omega.$
\mlabel{thm:pp1}
\end{theorem}

\begin{proof}
For $x,y\in R$ and $\alpha,\beta\in \Omega,$ we have
\begin{align*}
P(x\ot\alpha)P(y\ot\beta)=& \ \bigl(P_\alpha(x)\ot\alpha\bigr)(P_\beta(y)\ot \beta)
= \bigl(P_\alpha(x)P_\beta(y)\bigr)\ot\alpha\beta\\
=& \ P_{\alpha\beta}\bigl(P_\alpha(x)y+xP_\beta(y)+\lambda xy\bigr)\ot \alpha\beta\\
=& \ P\Bigl(\bigl(P_\alpha(x)y+xP_\beta(y)+\lambda xy\bigr)\ot\alpha\beta\Bigr)\\
=& \ P\bigl(P_\alpha(x)y\ot\alpha\beta+xP_\beta(y)\ot\alpha\beta+\lambda xy\ot\alpha\beta\bigr)\\
=& \ P\Bigl(\bigl(P_\alpha(x)\ot\alpha\bigr)(y\ot \beta)+(x\ot\alpha)\bigl(P_\beta(y)\ot\beta\bigr)+\lambda(x\ot\alpha)(y\ot\beta)\Bigr)\\
=& \ P\bigl(P(x\ot\alpha)(y\ot\beta)+(x\ot\alpha)P(y\ot\beta)+\lambda(x\ot\alpha)(y\ot\beta)\bigr),
\end{align*}
as required.
\end{proof}

\subsection{Dendriform family algebras and tridendriform family algebras}
The concept of dendriform family algebras was proposed in~\mcite{ZG}, as a generalization of dendriform algebras invented by Loday~\mcite{Lod93} in the study of algebraic $K$-theory.

\begin{defn}\cite{Lod93}
A {\bf dendriform algebra} is a $\bfk$-module $D$ with two binary operations $\prec, \succ$ such that for $ x, y, z\in D$,
\begin{align*}
(x\prec y) \prec z=\ & x \prec (y\prec z+y \succ z),  \\
(x\succ y)\prec z=\ & x\succ (y\prec z), \\
(x\prec y+x\succ y) \succ z  = \ & x\succ(y \succ z). 
\end{align*}
\mlabel{defn:dd}
\end{defn}

\begin{defn}\cite{ZG}
Let $\Omega$ be a semigroup. A {\bf dendriform family algebra} is a $\bfk$-module $D$ with a family of binary operations $\{\prec_{\omega},\succ_\omega\,\mid \omega \in \Omega\}$ such that for $ x, y, z\in D$ and $\alpha,\beta\in \Omega$,
\begin{align}
(x\prec_{\alpha} y) \prec_{\beta} z=\ & x \prec_{\bim{\alpha \beta}} (y\prec_{\beta} z+y \succ_{\alpha} z), \mlabel{eq:ddf1} \\
(x\succ_{\alpha} y)\prec_{\beta} z=\ & x\succ_{\alpha} (y\prec_{\beta} z),\mlabel{eq:ddf2} \\
(x\prec_{\beta} y+x\succ_{\alpha} y) \succ_{\bim{\alpha \beta}}z  =\ & x\succ_{\alpha}(y \succ_{\beta} z). \mlabel{eq:ddf3}
\end{align}
\mlabel{defn:dend}
\end{defn}

The concept of tridendriform family algebras was also introduced in~\cite{ZG}, which is a generalization of tridendriform algebras invented by Loday and Ronco~~\mcite{LoRo04} in the study of polytopes and Koszul duality.

\begin{defn}\cite{LoRo04}
 A {\bf tridendriform  algebra} is a $\bfk$-module $T$ equipped with
three binary operations $\prec, \succ$
and $\cdot$ such that for $ x, y, z\in T$,
\begin{align*}
(x\prec y)\prec z=\ & x\prec (y\prec z+y \succ z + y \cdot z), \\
(x\succ y)\prec z=\ & x\succ(y\prec z), \\
(x\prec y + x\succ y + x \cdot y) \succ z  = \ & x\succ (y\succ z),  \\
(x\succ y)\cdot  z=\  & x\succ(y \cdot  z), \\
(x\prec y)\cdot  z= \ & x \cdot (y\succ z), \\ 
(x\cdot y)\prec z= \ & x\cdot (y\prec z), \\ 
(x\cdot y)\cdot  z= \ & x\cdot  (y\cdot  z).
\end{align*}
\mlabel{defn:td}
\end{defn}

\begin{defn}\cite{ZG}
Let $\Omega$ be a semigroup. A {\bf tridendriform family algebra} is a $\bfk$-module $T$ equipped with
a family of binary operations $\{\prec_{\omega}, \succ_\omega \,\mid \omega \in \Omega\}$
and a binary operation $\cdot$ such that for $ x, y, z\in T$ and $\alpha,\beta\in \Omega$,
\begin{align}
(x\prec_{\alpha} y)\prec_{\beta} z=\ & x\prec_{\bim{\alpha \beta}} (y\prec_{\beta} z+y \succ_{\alpha} z + y \cdot z), \mlabel{eq:tdf1}\\
(x\succ_{\alpha} y)\prec_{\beta} z=\ & x\succ_{\alpha}(y\prec_{\beta} z), \mlabel{eq:tdf2}\\
(x\prec_{\beta} y + x\succ_{\alpha} y + x \cdot y) \succ_{\bim{\alpha \beta}}z  = \ & x\succ_{\alpha} (y\succ_{\beta} z), \mlabel{eq:tdf3} \\
(x\succ_{\alpha} y)\cdot  z=\  & x\succ_{\alpha}(y \cdot  z), \mlabel{eq:tdf4}\\
(x\prec_{\alpha} y)\cdot  z= \ & x \cdot (y\succ_{\alpha} z), \mlabel{eq:tdf5}\\ 
(x\cdot y)\prec_{\alpha} z= \ & x\cdot (y\prec_{\alpha} z), \mlabel{eq:tdf6}\\ 
(x\cdot y)\cdot  z= \ & x\cdot  (y\cdot  z).\mlabel{eq:tdf7}
\end{align}
\end{defn}

\begin{remark}
When the semigroup $\Omega$ is taken to be the trivial monoid with one single element, a dendriform (resp.~tridendriform) family algebra is precisely a
dendriform (resp.~tridendriform) algebra.
\end{remark}

Let $\mathfrak{RBF_\lambda}$ be the category of Rota-Baxter family algebras of weight $\lambda$,
and let $\mathfrak{{DDF}}$ (resp. $\mathfrak{{DTF}}$) be the category of dendriform (resp. tridendriform) family algebras.
A functor $\varepsilon:\mathfrak{RBF_\lambda}\ra \mathfrak{DTF}$ has been introduced in~\cite[Theorem 4.4]{ZG}. Further we have

\begin{prop}
Let $\Omega$ be a semigroup.
\begin{enumerate}
\item  A Rota-Baxter family algebra $(R, \,\{P_{\omega} \mid \omega \in \Omega\} )$ of weight $\lambda$ induces a dendriform family algebra
$(R, \, \{\prec_{{\omega}}, \succ_{{\omega}}  \, \mid \omega \in \Omega\})$, where
\begin{equation*}
x\prec_{{\omega}}y := xP_{\omega}(y)+\lambda xy\,\text{ and }\, x\succ_{{\omega}}y := P_{\omega}(x)y,\,\text{ for }\, x, y \in R.
\mlabel{eq:RBTD3}
\end{equation*}
\mlabel{it:rbtodend}

\item A tridendriform family algebra $(T,\{\prec_\omega,\succ_\omega\mid\omega\in\Omega\},\cdot)$ induces a
dendriform family algebra $(T,\{\prec'_\omega,\succ'_\omega\mid\omega\in\Omega\})$, where
  $$x \prec'_\omega y:= x \prec_\omega y + x \cdot y \,\text{ and }\ x \succ'_\omega y:= x \succ_\omega y, \,\text{ for }\, x, y \in T.$$
\mlabel{it:tritodend}
\end{enumerate}
\mlabel{prop:RBTD3}
\end{prop}

\begin{proof}
\mref{it:rbtodend} For any $x,y,z\in R$ and $\alpha,\beta\in\Omega$,
\begin{align*}
(x\prec_\alpha y)\prec_\beta z=& \ (xP_\alpha(y)+\lambda xy)\prec_\beta z\\
=& \ (xP_\alpha(y)+\lambda xy)P_\beta(z)+\lambda(xP_\alpha(y)+\lambda xy)z\\
=& \ xP_\alpha(y)P_\beta(z)+\lambda xyP_\beta(z)+\lambda xP_\alpha(y)z+\lambda^{2}xyz\\
=& \ xP_{\alpha\beta}(yP_\beta(z)+P_\alpha(y)z+\lambda yz)+\lambda x(yP_\beta(z)+P_\alpha(y)z+\lambda yz)\\
=& \ x\prec_{\alpha\beta}(yP_\beta(z)+P_\alpha(y)z+\lambda yz)\\
=& \ x\prec_{\alpha\beta}(y\prec_\beta z+y\succ_\alpha z),\\
(x\succ_\alpha y)\prec_\beta z=& \ P_\alpha(x)y\prec_\beta z=P_\alpha(x)yP_\beta(z)+\lambda P_\alpha(x)yz\\
=& \ P_\alpha(x)(yP_\beta(z)+\lambda yz)=P_\alpha(x)(y\prec_\beta z)\\
=& \ x\succ_\alpha(y\prec_\beta z),\\
x\succ_\alpha(y\succ_\beta z)=& \ x\succ_\alpha(P_\beta(y)z)=P_\alpha(x)(P_\beta(y)z)=(P_\alpha(x)P_\beta(y))z\\
=& \ P_{\alpha\beta}(P_\alpha(x)y+xP_\beta(y)+\lambda xy)z\\
=& \ (P_\alpha(x)y+xP_\beta(y)+\lambda xy)\succ_{\alpha\beta} z\\
=& \ (x\succ_\alpha y+x\prec_\beta y)\succ_{\alpha\beta} z.
\end{align*}

\mref{it:tritodend} The proof is straightforward.
\end{proof}

Let
$\eta:\mathfrak{RBF_\lambda}\ra \mathfrak{DDF} \,\text{ and }\, \gamma:\mathfrak{DTF} \ra \mathfrak{DDF}$
be the functors obtained in Proposition~\mref{prop:RBTD3}.
Then we have the commutative diagram:
\[
\xymatrix{
  \mathfrak{RBF_\lambda} \ar[dr]_{\eta} \ar[r]^{\varepsilon}
                & \mathfrak{{DTF}} \ar[d]^{\gamma}  \\
                &  \mathfrak{{DDF}}            }
\]

The following result precises the link between (tri)dendriform family algebras and ordinary (tri)dendriform algebras.

\begin{theorem}
\begin{enumerate}
\item Let $(D, \{\prec_\omega,\succ_\omega\mid\omega\in\Omega\})$ be a dendriform family. Then $(D\ot\bfk\Omega,\prec,\succ)$ is a dendriform algebra, where
\begin{align*}
(x\ot\alpha)\prec(y\ot\beta):=&(x\prec_\beta y)\ot\alpha\beta\\
 (x\ot\alpha)\succ(y\ot\beta):=& (x\succ_\alpha y)\ot\alpha\beta,\,
\text{ for }\, x,y\in D \text{ and } \alpha,\beta\in \Omega.
\end{align*}
\mlabel{item:prsu1}

\item Let $(T, \{\prec_\omega, \succ_\omega\mid\omega\in\Omega\},\cdot)$ be a tridendriform family. Then $(T\ot\bfk\Omega,\prec,\succ,\bullet)$ is a tridendriform algebra , where
\begin{align*}
(x\ot\alpha)\prec(y\ot\beta):=& (x\prec_\beta y)\ot\alpha\beta,\\
(x\ot\alpha)\succ(y\ot\beta):= & (x\succ_\alpha y)\ot\alpha\beta, \\
(x\ot\alpha)\bullet(y\ot\beta):=&(x\cdot y)\ot\alpha\beta,\,\text{ for }\,x,y\in T\,\text{ and }\, \alpha,\beta\in \Omega.
\end{align*}
\mlabel{item:prsut}
\end{enumerate}
\mlabel{thm:prsut}
\end{theorem}

\begin{proof}
We only prove Item~\mref{item:prsut}, as the proof of Item~\mref{item:prsu1} is similar and easier. For $x\ot\alpha, y\ot\beta, z\ot\gamma \in T\ot\bfk\Omega,$ we get
\begin{align*}
\Big((x\ot\alpha)\prec(y\ot\beta)\Big)\prec(z\ot\gamma)=& \ \Big((x\prec_\beta y)\ot\alpha\beta\Big)\prec(z\ot\gamma)\\
=& \ \Big((x\prec_\beta y)\prec_\gamma z\Big)\ot\alpha\beta\gamma\\
=& \ \Big(x\prec_{\beta\gamma}(y\prec_\gamma z+y\succ_\beta z+y\cdot z)\Big)\ot\alpha\beta\gamma\quad(\text{by Eq.~(\mref{eq:tdf1}}))\\
=& \ (x\ot\alpha)\prec\Big((y\prec_\gamma z+y\succ_\beta z+y\cdot z)\ot\beta\gamma\Big)\\
=& \ (x\ot\alpha)\prec\Big((y\prec_\gamma z)\ot\beta\gamma+(y\succ_\beta z)\ot\beta\gamma+(y\cdot z)\ot\beta\gamma\Big)\\
=& \ (x\ot\alpha)\prec\Big((y\ot\beta)\prec(z\ot\gamma)+(y\ot\beta)\succ(z\ot\gamma)+(y\ot\beta)\bullet(z\ot\gamma)\Big),\\
\Big((x\ot\alpha)\succ(y\ot\beta)\Big)\prec(z\ot\gamma)=& \ \Big((x\succ_\alpha y)\ot\alpha\beta\Big)\prec(z\ot\gamma)
= \Big((x\succ_\alpha y)\prec_\gamma z\Big)\ot\alpha\beta\gamma\\
=& \ \Big(x\succ_\alpha(y\prec_\gamma z)\Big)\ot\alpha\beta\gamma\quad(\text{by Eq.~(\mref{eq:tdf2}}))\\
=& \ (x\ot\alpha)\succ\Big((y\prec_\gamma z)\ot\beta\gamma\Big)\\
=& \ (x\ot\alpha)\succ\Big((y\ot\beta)\prec(z\ot\gamma)\Big),\\
(x\ot\alpha)\succ\Big((y\ot\beta)\succ(z\ot\gamma)\Big)=& \ (x\ot\alpha)\succ\Big((y\succ_\beta z)\ot\beta\gamma\Big)\\
=& \ \Big(x\succ_\alpha(y\succ_\beta z)\Big)\ot\alpha\beta\gamma\\
=& \ \Big((x\prec_\beta y+x\succ_\alpha y+x\cdot y)\succ_{\alpha\beta}z\Big)\ot\alpha\beta\gamma\quad(\text{by Eq.~(\mref{eq:tdf3}}))\\
=& \ \Big((x\prec_\beta y+x\succ_\alpha y+x\cdot y)\ot\alpha\beta\Big)\succ(z\ot\gamma)\\
=& \ \Big((x\prec_\beta y)\ot\alpha\beta+(x\succ_\alpha y)\ot\alpha\beta+(x\cdot y)\ot\alpha\beta\Big)\succ(z\ot\gamma)\\
=& \ \Big((x\ot\alpha)\prec(y\ot \beta)+(x\ot\alpha)\succ(y\ot\beta)+(x\ot\alpha)\bullet(y\ot\beta)\Big)\succ(z\ot\gamma),\\
\Big((x\ot\alpha)\succ(y\ot\beta)\Big)\bullet(z\ot\gamma)=& \ \Big((x\succ_\alpha y)\ot\alpha\beta\Big)\bullet(z\ot\gamma)\\
=& \ \Big((x\succ_\alpha y)\cdot z\Big)\ot\alpha\beta\gamma
= \Big(x\succ_\alpha(y\cdot z)\Big)\ot\alpha\beta\gamma\quad(\text{by Eq.~(\mref{eq:tdf4}}))\\
=& \ (x\ot\alpha)\succ\Big((y\cdot z)\ot\beta\gamma\Big)\\
=& \ (x\ot\alpha)\succ\Big((y\ot\beta)\bullet(z\ot\gamma)\Big),\\
\Big((x\ot\alpha)\prec(y\ot\beta)\Big)\bullet(z\ot\gamma)=& \ \Big((x\prec_\beta y)\ot\alpha\beta\Big)\bullet(z\ot\gamma)
=\Big((x\prec_\beta y)\cdot z\Big)\ot\alpha\beta\gamma\\
=& \ \Big(x\cdot(y\succ_\beta z)\Big)\ot\alpha\beta\gamma\quad(\text{by Eq.~(\mref{eq:tdf5}}))\\
=& \ (x\ot\alpha)\bullet\Big((y\succ_\beta z)\ot\beta\gamma\Big)\\
=& \ (x\ot\alpha)\bullet\Big((y\ot\beta)\succ(z\ot\gamma)\Big),\\
\Big((x\ot\alpha)\bullet(y\ot\beta)\Big)\prec(z\ot\gamma)=& \ \Big((x\cdot y)\ot\alpha\beta\Big)\prec(z\ot\gamma)
=\Big((x\cdot y)\prec_\gamma z\Big)\ot\alpha\beta\gamma\\
=& \ \Big(x\cdot(y\prec_\gamma z)\Big)\ot\alpha\beta\gamma\quad(\text{by Eq.~(\mref{eq:tdf6}}))\\
=& \  (x\ot \alpha)\bullet\Big((y\prec_\gamma z)\ot\beta\gamma\Big)\\
=& \ (x\ot\alpha)\bullet\Big((y\ot\beta)\prec(z\ot\gamma)\Big),\\
\Big((x\ot\alpha)\bullet(y\ot\beta)\Big)\bullet(z\ot\gamma)=& \ \Big((x\cdot y)\ot\alpha\beta\Big)\bullet(z\ot\gamma)
=\Big((x\cdot y)\cdot z\Big)\ot\alpha\beta\gamma\\
=& \ \Big(x\cdot(y\cdot z)\Big)\ot\alpha\beta\gamma\quad(\text{by Eq.~(\mref{eq:tdf7}}))\\
=& \ (x\ot\alpha)\bullet\Big((y\cdot z)\ot\beta\gamma\Big)\\
=& \ (x\ot\alpha)\bullet\Big((y\ot\beta)\bullet(z\ot\gamma)\Big).
\end{align*}
This completes the proof.
\end{proof}

\section{ Free dendriform family algebras}
\mlabel{sec:dd}
In this section, we construct free dendriform family algebras.
For this, let us first briefly recall the construction of free dendriform algebras.
For details, see~\cite{Lod93, LoRo98, Lody02, Ron02}.

\subsection{Free dendriform algebras}
\mlabel{sub:fdd}
Let $X$ be a set. For $n\geq 0$, let $Y_{n,\,X}$ be the set of planar binary trees with $n+1$ leaves and with internal vertices decorated by elements of $X$. The unique tree with one leaf is denoted by $|$. So we have $Y_{0,\,X}=\{|\}$.
Here are the first few of them.
\begin{align*}
Y_{0,\,X}&=\{|\},\ \
Y_{1,\,X}=\left\{ \stree x\Bigm| x\in X \right\},\ \
Y_{2,\,X}=\left\{
\XX{\xxr{-5}5
\xxhu00x \xxhu{-5}5y
}, \,
\XX{\xxl55
\xxhu00x \xxhu55y
}\Bigm| x,y\in X
\right\},\\
Y_{3,\,X}&=\left\{
\XX[scale=1.6]{\xxr{-4}4\xxr{-7.5}{7.5}
\xxhu00{x} \xxhu[0.1]{-4}4{y} \xxhu[0.1]{-7.5}{7.5}{z}
}, \,
\XX[scale=1.6]{\xxl44\xxl{7.5}{7.5}
\xxhu00{x} \xxhu[0.1]44{y} \xxhu[0.1]{7.5}{7.5}{z}
}, \,
\XX[scale=1.6]{\xxr{-6}6\xxl66
\xxhu00{x} \xxhu[0.1]66{y} \xxhu{-6}6{z}
}, \,
\XX[scale=1.6]{\xxr{-5}5\xxl{-2}8
\xxhu00x
\xxhu[0.1]{-5}5{y} \xxhu[0.1]{-2}8{z}
}, \,
\XX[scale=1.6]{\xxl55\xxr28
\xxhu00x
\xxhu[0.1]55{\,y} \xxhu[0.1]28{z}
}\, \Biggm| x,y,z\in X\right\}.
\end{align*}


For $T\in Y_{m,\,X}, U\in Y_{n,\,X}$ and $x\in X$, the grafting $\vee_x$ of $T$ and $U$ over the vertex $x$ is defined to be the planar binary tree $T\vee_x U\in Y_{m+n+1,\,X}$ obtained by adding a new vertex decorated by $x$ and joining the roots of $T$ and $U$ to the new vertex.\\

\begin{remark}
In the graphical representation above, the edge pointing downwards is the root, the upper edges are the leaves. The other edges, joining two internal vertices, are called internal edges.
\end{remark}
\begin{exam}
Let $T=|$ and $U= \stree y$ with $y\in X$. Then
$$T\vee_x U=|\vee_x \stree y=\XX{\xxl55
\xxhu00x \xxhu55y
}.$$
\end{exam}

Given a planar binary tree $T\in Y_{n,\,X}$ not equal to $|$, there is a unique decomposition $T=T^{l}\vee_x T^{r}$ for some $x\in X$. For example,
$$ \stree x=|\vee_x|,\, \XX{\xxr{-5}5
\xxhu00x \xxhu{-5}5y
}=\stree y\vee_x|,\, \XX{\xxl55
\xxhu00x \xxhu55y
}=|\vee_x\stree y.$$

Let $\mathrm{DD}(X):=\underset{n\geq 1}\bigoplus\,\bfk Y_{n,\,X}$. Define binary operations $\prec$ and $\succ$ on $\mathrm{DD}(X)$
recursively by
\begin{enumerate}
\item $|\succ T:=T\prec |:=T\,\text{ and }\, |\prec T:=T\succ |:=0$ for $T\in Y_{n,\,X}$ with $n\geq 1.$
\item For $T=T^{l}\vee_{x} T^{r}$ and $U=U^{l}\vee_{y} U^{r},$ define
$$T\prec U:=T^{l}\vee_{x} (T^{r}\prec U+T^{r}\succ U),\quad T\succ U:=(T\prec U^{l}+T\succ U^{l})\vee_{y} U^{r}.$$
\end{enumerate}

Let $j$ be the unique linear map from $X$ into $\mathrm{DD}(X)$
defined by $j(x)=\stree x$ for $x\in X$. The following result is well-known.

\begin{theorem}\cite{LoRo98}
Let $X$ be a set. Then
$(\mathrm{DD}(X),\prec,\succ),$ together with the map $j$, is the free dendriform algebra on $X.$
\end{theorem}

\subsection{Free dendriform family algebras}\mlabel{sub:dd}
In this subsection, we apply typed decorated planar binary trees
to construct free dendriform family algebras.
For this, let us first recall typed decorated rooted trees studied in~\cite{BHZ,Foi18}.

For a rooted tree $T$, denote by $V(T)$ (resp. $E(T)$) the set of its vertices (resp. edges).

\begin{defn}\cite{BHZ}
Let $X$ and $\Omega$ be two sets. An {\bf $X$-decorated $\Omega$-typed (abbreviated typed decorated) rooted tree}
is a triple $T = (T,\dec, \type)$, where
\begin{enumerate}
\item $T$ is a rooted tree.
\item $\dec: V(T)\ra X$ is a map.
\item $\type: E(T)\ra \Omega$ is a map.
\end{enumerate}
\mlabel{defn:tdtree}
\end{defn}
In other words, vertices of $T$ are decorated by elements of $X$ and edges of $T$ are decorated by elements of $\Omega$.\\

As we can see from the examples drawn above, the graphical representation of planar binary trees is slightly different from the graphical representation of rooted trees used in \cite{BHZ,Foi18}: the root and the leaves are now edges rather than vertices. Here the set $V(T)$ must be replaced by the set $IV(T)$ of internal vertices of $T$, and the set $E(T)$ must be replaced by the set $IE(T)$ of internal edges, i.e. edges which are neither a leaf nor the root. Hence we propose the following definition:

\begin{defn}
Let $X$ and $\Omega$ be two sets. An {\bf $X$-decorated $\Omega$-typed planar binary tree} is a triple $T = (T,\dec, \type)$, where
\begin{enumerate}
\item $T$ is a planar binary tree.
\item $\dec: IV(T)\ra X$ is a map.
\item $\type: IE(T)\ra \Omega$ is a map.
\end{enumerate}
\mlabel{defn:type}
\end{defn}

Note that in the definition of dendriform family algebras, $\Omega$ is a semigroup.
However in the following construction of free dendriform family algebras, if the semigroup $\Omega$ has no identity element, it will be convenient to consider the monoid $\Omega^{1}:=\Omega\sqcup \{1\}$
obtained from $\Omega$ by adjoining an identity:
 $$1\omega:=\omega1:=\omega,\,\text{ for }\,\omega\in\Omega\,\text{ and }\, 11:=1.$$

Let $X$ be a set and let $\Omega$ be a semigroup. For $n\geq 0,$
let $Y_{n}:=Y_{n,\,X,\,\Omega}$ be the set of $X$-decorated $\Omega^{1}$-typed planar binary trees with $n+1$ leaves,
such that leaves are decorated by the identity 1 in $\Omega^{1}$ and internal edges are decorated by elements of $\Omega.$ 
Let us expose some examples for better understanding.
For convenience, we omit the decoration 1 in the sequel.
\begin{align*}
Y_0&=\{|\},\ \
Y_1=\left\{\stree x\Bigm|x\in X
\right\},\ \
Y_2=\left\{
\XX{\xxr{-5}5
\node at (-0.4,0) {$\alpha$};
\xxhu00x \xxhu{-5}5y
}, \,
\XX{\xxl55
\node at (0.4,0) {$\alpha$};
\xxhu00x \xxhu55y
}\Bigm|x,y\in X,\alpha\in\Omega
\right\},\\
Y_3&=\left\{
\XX[scale=1.6]{\xxr{-4}4\xxr{-7.5}{7.5}
\node at (-0.35,0.1) {$\alpha$};
\node at (-0.75,0.4) {$\beta$};
\xxhu00{x} \xxhu[0.1]{-4}4{y} \xxhu[0.1]{-7.5}{7.5}{z}
}, \,
\XX[scale=1.6]{\xxl44\xxl{7.5}{7.5}
\node at (0.35,0.1) {$\alpha$};
\node at (0.7,0.4) {$\beta$};
\xxhu00{x} \xxhu[0.1]44{y} \xxhu[0.1]{7.5}{7.5}{z}
}, \,
\XX[scale=1.6]{\xxr{-6}6\xxl66
\node at (-0.5,0.15) {$\beta$};
\node at (0.45,0.15) {$\alpha$};
\xxhu00{x} \xxhu66{y} \xxhu{-6}6{z}
}, \,
\XX[scale=1.6]{\xxr{-5}5\xxl{-2}8
\node at (-0.4,0.1) {$\alpha$};
\node at (-0.2,0.55) {\tiny $\beta$};
\xxhu00x
\xxhu[0.1]{-5}5{y} \xxhu[0.1]{-2}8{z}
}, \,
\XX[scale=1.6]{\xxl55\xxr28
\node at (0.45,0.15) {$\alpha$};
\node at (0.22,0.5) {\tiny $\beta$};
\xxhu[0.1]00{x\,}
\xxhu[0.1]55{\,y} \xxhu[0.1]28{z}
},\ldots \Bigg|\,x,y,z\in X,\alpha,\beta\in\Omega\right\}.
\end{align*}

For $T\in Y_{m}, U\in Y_{n}$, $x\in X$ and $\alpha,\beta\in\Omega^{1},$ the grafting $\Ve x,\alpha,\beta;$ of $T$ and $U$ over $x$ and $(\alpha,\beta)$ is defined to be the tree $T\Ve x,\alpha,\beta; U\in Y_{m+n+1}$, obtained by adding a new vertex decorated by $x$ and joining the roots of $T$ and $U$ with the new vertex via two new edges decorated by $\alpha$ and $\beta$ respectively.
The {\bf depth} $\dep{(T)}$ of a rooted tree $T$ is the maximal length of linear chains of vertices from the root to the leaves of the tree.
For example,
$$\dep\biggl(\stree x\biggr) = 1\, \text{ and } \, \dep\biggl(\XX{\xxr{-5}5
\node at (-0.4,0) {$\alpha$};
\xxhu00x \xxhu{-5}5y
}\biggr) = 2.$$

\begin{exam}
Let $T=\stree x$ and $U= \stree y$ with $x,y\in X$. For $z\in X$ and $\alpha,\beta\in\Omega,$
$$T\Ve z,\alpha,\beta; U=\stree x\Ve z,\alpha,\beta; \stree y=\XX{\xxr{-6}6\xxl66
\node at (-0.5,0.1) {$\alpha$};
\node at (0.55,0.1) {$\beta$};
\xxhu00{z} \xxhu66{y} \xxhu{-6}6{x}
}.$$
\end{exam}

Given a typed decorated planar binary tree $T\in Y_n$ not equal to  $|$, there is a unique decomposition
$$T=T^{l}\Ve x,\alpha,\beta; T^{r}\,\text{ for some }\, x\in X\,\text{ and }\, \alpha,\beta\in \Omega^{1}.$$
For example,
\begin{equation*}
 \stree x=|\Ve x,1,1;|,\, \XX{\xxr{-5}5
\xxhu00x \xxhu{-5}5y
\node at (-0.5,-0.1) {$\alpha$};
}=\stree y\Ve x,\alpha,1;|,\, \XX{\xxl55
\xxhu00x \xxhu55y
\node at (0.45,0) {$\alpha$};
}=|\Ve x,1,\alpha;\stree y,
\end{equation*}
\begin{equation*}
\XX{\xxr{-6}6\xxl66
\node at (-0.5,0.1) {$\alpha$};
\node at (0.55,0.1) {$\beta$};
\xxhu00{x} \xxhu66{z} \xxhu{-6}6{y}
}=\stree y\Ve x,\alpha,\beta;\stree z.
\end{equation*}
Denote by
$$\DDF:=\underset{n\geq 1}\bigoplus\,\bfk Y_{n}.$$

\begin{defn}\label{def:freeDF}
Let $X$ be a set and let $\Omega$ be a semigroup. We define binary operations $\{\prec_\omega,\succ_\omega\mid\omega\in\Omega\}$ on $\DDF$ recursively on $\dep(T)+\dep(U)$ by
\begin{enumerate}
\item $|\succ_\omega T:=T\prec_\omega |:=T\,\text{ and }\, |\prec_\omega T:=T\succ_\omega |:=0$ for $\omega\in\Omega$ and $T\in Y_{n}$ with $n\geq 1$.\mlabel{it:intial1}

\item For $T=T^{l}\Ve x,\alpha_1,\alpha_2; T^{r}$ and $U=U^{l}\Ve y,\beta_1,\beta_2; U^{r},$ define
\begin{align}
T\prec_\omega U:=& \ T^{l}\Ve x,\alpha_1,\alpha_2\omega; (T^{r}\prec_\omega U+T^{r}\succ_{\alpha_2} U),\mlabel{eq:recur1}\\
T\succ_\omega U:=& \ (T\prec_{\beta_1} U^{l}+T\succ_\omega U^{l})\Ve y,\omega\beta_1,\beta_2; U^{r},\,\text{ where }\, \omega \in \Omega.\mlabel{eq:recur}
\end{align}
\end{enumerate}
\mlabel{defn:recur}
\end{defn}

In the following, we employ the convention that
\begin{equation}
|\succ_1 T:=T\prec_1 |:=T\,\text{ and }\, |\prec_1 T:=T\succ_1 |:=0.
\mlabel{eq:conv}
\end{equation}

\begin{exam}
Let $T=\stree x$ and $U=\stree y$ with $x,y\in X$. For $\omega\in \Omega,$ we have
\begin{align*}
\stree x\prec_\omega \stree y
=& \ (|\Ve x,1,1;|)\prec_\omega \stree y\\
=& \ |\Ve x,1,\omega;\biggl(|\prec_\omega \stree y+|\succ_1\stree y\biggr)\quad(\text{by Eq.~(\mref{eq:recur1})})\\
=& \ |\Ve x,1,\omega;\stree y\quad(\text{by Item~\mref{it:intial1} of Definition~\mref{defn:recur} and Eq.~(\mref{eq:conv})})\\
=& \ \XX{\xxl55
\node at (0.4,0) {$\omega$};
\xxhu00x \xxhu55y
},\\
\stree x\succ_\omega \stree y
=& \ \stree x\succ_\omega (|\Ve y,1,1;|)\\
=& \ \biggl(\stree x\prec_1 |+\stree x\succ_\omega |\biggr)\Ve y,\omega,1;|\quad(\text{by Eq.~(\mref{eq:recur})})\\
=& \ \stree x\Ve y,\omega,1;|\quad(\text{by Item~\mref{it:intial1} of Definition~\mref{defn:recur} and Eq.~(\mref{eq:conv})})\\
=& \ \XX{\xxr{-5}5
\node at (-0.45,0) {$\omega$};
\xxhu00y \xxhu{-5}5{x}
}.
\end{align*}
\mlabel{exam:tu}
\end{exam}

\begin{exam}
Let $T=\XX{\xxr{-6}6\xxl66
\node at (-0.5,0.1) {$\alpha$};
\node at (0.55,0.1) {$\beta$};
\xxhu00{x} \xxhu66{u} \xxhu{-6}6{z}
}$ and $U=\stree y$ with $x, y,z,u\in X$ and $\alpha,\beta\in\Omega.$ For $\omega\in\Omega,$
\begin{align*}
T\prec_\omega U=& \ \biggl(\stree z\Ve x,\alpha,\beta;\stree u\biggr)\prec_\omega \stree y\\
=& \ \stree z\Ve x,\alpha,\beta\omega;\biggl(\stree u\prec_\omega\stree y +\stree u\succ_{\beta}\stree y\biggr)\\
=& \ \stree z\Ve x,\alpha,\beta\omega;\biggl(\XX{\xxl55
\node at (0.4,0) {$\omega$};
\xxhu00u \xxhu55y
}+\XX{\xxr{-5}5
\node at (-0.5,-0.1) {$\beta$};
\xxhu00y \xxhu{-5}5{u}
}\biggr)\quad(\text{by Example}~\mref{exam:tu})\\
=& \ \XX[scale=1.6]{\xxl44\xxl{7.5}{7.5}\xxr{-7}7
\node at (-0.6,0.3) {$\alpha$};
\node at (0.4,0) {$\beta\omega$};
\node at (0.75,0.45) {$\omega$};
\xxhu00{x} \xxhu[0.1]44{u} \xxhu[0.1]{7.5}{7.5}{y} \xxhu[0.1]{-7}7{z}
}+\XX[scale=1.6]{\xxl55\xxr28\xxr{-7}7
\node at (-0.6,0.3) {$\alpha$};
\node at (0.5,0.1) {$\beta\omega$};
\node at (0.22,0.5) {\tiny$\beta$};
\xxhu[0.1]00{x\,}
\xxhu[0.1]55{\,y} \xxhu[0.1]28{u} \xxhu{-7}7{z}
},\\
T\succ_\omega U=& \ \XX{\xxr{-6}6\xxl66
\node at (-0.5,0.1) {$\alpha$};
\node at (0.55,0.1) {$\beta$};
\xxhu00{x} \xxhu66{u} \xxhu{-6}6{z}
}\succ_\omega (|\Ve y,1,1;|)
= \biggl(\XX{\xxr{-6}6\xxl66
\node at (-0.5,0.1) {$\alpha$};
\node at (0.55,0.1) {$\beta$};
\xxhu00{x} \xxhu66{u} \xxhu{-6}6{z}
}\prec_1|+\XX{\xxr{-6}6\xxl66
\node at (-0.5,0.1) {$\alpha$};
\node at (0.55,0.1) {$\beta$};
\xxhu00{x} \xxhu66{u} \xxhu{-6}6{z}
}\succ_\omega |\biggr)\Ve y,\omega,1;|\\
=& \ \XX{\xxr{-6}6\xxl66
\node at (-0.5,0.1) {$\alpha$};
\node at (0.55,0.1) {$\beta$};
\xxhu00{x} \xxhu66{u} \xxhu{-6}6{z}
}\Ve y,\omega,1;|\quad(\text{by Item~\mref{it:intial1} of Definition~\mref{defn:recur} and Eq.~(\mref{eq:conv})})\\
=& \ \XX[scale=1.6]{\xxr{-5}5\xxl{-2}8\xxr{-8}{8}
\node at (-0.4,0.1) {$\omega$};
\node at (-0.2,0.55) {\tiny $\beta$};
\node at (-0.75,0.5) {$\alpha$};
\xxhu00y
\xxhu[0.1]{-5}5{x} \xxhu[0.1]{-2}8{u}\xxhu[0.1]{-8}{8}{z}
}.
\end{align*}
\end{exam}

\begin{prop}
Let $X$ be a set and let $\Omega$ be a semigroup.
Then $(\DDF, \{\prec_\omega,\succ_\omega\mid \omega\in \Omega\})$ is a dendriform family algebra, generated by $\biggl\{\stree x\Bigm| x\in X\biggr\}.$
\mlabel{prop:dend}
\end{prop}

\begin{proof}
We proceed in two steps to prove the result.

{\noindent \bf Step 1:} We first prove $(\DDF, \{\prec_\omega,\succ_\omega\mid \omega\in \Omega\})$ is a dendriform family algebra. Let $$T=T^{l}\Ve x,\alpha_1,\alpha_2; T^{r}, U=U^{l}\Ve y,\beta_1,\beta_2; U^{r}, W=W^{l}\Ve z,\gamma_1,\gamma_2; W^{r}\in \DDF.$$
 Then we apply induction on $\dep(T)+\dep(U)+\dep(W)\geq 3.$ For the initial step $\dep(T)+\dep(U)+\dep(W)=3$, we have $$T=\stree x, U=\stree y\,\text{ and }\, W=\stree z$$ and so
\begin{align*}
&(T\prec_\beta U)\prec_\gamma W\\
=& \ \biggl(\stree x\prec_\beta\stree y\biggr)\prec_\gamma\stree z
= \biggl((|\Ve x,1,1;|)\prec_\beta \stree y\biggr)\prec_\gamma \stree z\\
=& \ \Biggl(|\Ve x,1,\beta;\biggl(|\prec_\beta \stree y+|\succ_1\stree y\biggr)\Biggr)\prec_\gamma \stree z\quad(\text{by Eq.~(\mref{eq:recur1})})\\
=& \ \biggl(|\Ve x,1,\beta;\stree y\biggr)\prec_\gamma \stree z\quad(\text{by Item~\mref{it:intial1} in Definition~\mref{defn:recur} and Eq.~(\mref{eq:conv}}))\\
=& \ |\Ve x,1,\beta\gamma;\biggl(\stree y\prec_\gamma \stree z+\stree y\succ_\beta\stree z\biggr)\quad(\text{by Eq.~(\mref{eq:recur1})})\\
=& \ |\Ve x,1,\beta\gamma;\biggl(\XX{\xxl55
\node at (0.4,0) {$\gamma$};
\xxhu00y \xxhu55z
}+\XX{\xxr{-5}5
\node at (-0.45,0) {$\beta$};
\xxhu00z \xxhu{-5}5{y}
}\biggr)\quad(\text{by Example~\mref{exam:tu}})\\
=& \ |\Ve x,1,\beta\gamma;\Biggl(|\prec_{\beta\gamma}\biggl(\XX{\xxl55
\node at (0.4,0) {$\gamma$};
\xxhu00y \xxhu55z
}+\XX{\xxr{-5}5
\node at (-0.45,0) {$\beta$};
\xxhu00z \xxhu{-5}5{y}
}\biggr)+|\succ_1\biggl(\XX{\xxl55
\node at (0.4,0) {$\gamma$};
\xxhu00y \xxhu55z
}+\XX{\xxr{-5}5
\node at (-0.45,0) {$\beta$};
\xxhu00z \xxhu{-5}5{y}
}\biggr)\Biggr)\\
&\quad(\text{by Item~\mref{it:intial1} in Definition~\mref{defn:recur} and Eq.~(\mref{eq:conv}}))\\
=& \ (|\Ve x,1,1;|)\prec_{\beta\gamma}\biggl(\XX{\xxl55
\node at (0.4,0) {$\gamma$};
\xxhu00y \xxhu55z
}+\XX{\xxr{-5}5
\node at (-0.45,0) {$\beta$};
\xxhu00z \xxhu{-5}5{y}
}\biggr)\quad(\text{by Eq.~(\mref{eq:recur1})})\\
=& \ \stree x \prec_{\beta\gamma}\biggl(\XX{\xxl55
\node at (0.4,0) {$\gamma$};
\xxhu00y \xxhu55z
}+\XX{\xxr{-5}5
\node at (-0.45,0) {$\beta$};
\xxhu00z \xxhu{-5}5{y}
}\biggr)\\
=& \ \stree x\prec_{\beta\gamma}\biggl(\stree y\prec_\gamma \stree z+\stree y\succ_\beta \stree z\biggr)\\
=& \ T\prec_{\beta\gamma}(U\prec_\gamma W+U\succ_\beta W),
\end{align*}
verifying Eq.~(\mref{eq:ddf1}). Next,
\begin{align*}
&(T\succ_\alpha U)\prec_\gamma W\\
=& \ \biggl(\stree x\succ_\alpha \stree y\biggr)\prec_\gamma \stree z
=\biggl(\stree x\succ_\alpha(|\Ve y,1,1;|)\biggr)\prec_\gamma \stree z\\
=& \ \Biggl(\biggl(\stree x\succ_\alpha |+\stree x\prec_1|\biggr)\Ve y,\alpha,1;|\Biggr)\prec_\gamma \stree z\quad(\text{by Eq.~(\mref{eq:recur})})\\
=& \ \biggl(\stree x\Ve y,\alpha,1;|\biggr)\prec_\gamma \stree z\quad(\text{by Item~\mref{it:intial1} in Definition~\mref{defn:recur} and Eq.~(\mref{eq:conv}}))\\
=& \ \stree x\Ve y,\alpha,\gamma;\biggl(|\prec_\gamma\stree z+|\succ_1\stree z\biggr)\quad(\text{by Eq.~(\mref{eq:recur1})})\\
=& \ \stree x\Ve y,\alpha,\gamma; \stree z\quad(\text{by Item~\mref{it:intial1} in Definition~\mref{defn:recur} and Eq.~(\mref{eq:conv})})\\
=& \ \biggl(\stree x\succ_\alpha |+\stree x\prec_1|\biggr)\Ve y,\alpha,\gamma;\stree z\quad(\text{by Item~\mref{it:intial1} in Definition~\mref{defn:recur} and Eq.~(\mref{eq:conv}}))\\
=& \ \stree x\succ_\alpha\biggl(|\Ve y,1,\gamma;\stree z\biggr)\quad(\text{by Eq.~(\mref{eq:recur})})\\
=& \ \stree x\succ_\alpha\Biggl(|\Ve y,1,\gamma;\biggl(|\prec_\gamma \stree z+|\succ_1\stree z\biggr)\Biggr)\quad(\text{by Item~\mref{it:intial1} in Definition~\mref{defn:recur} and Eq.~(\mref{eq:conv}}))\\
=& \ \stree x\succ_\alpha\biggl((|\Ve y,1,1;|)\prec_\gamma\stree z\biggr)
=\stree x\succ_\alpha\biggl(\stree y\prec_\gamma \stree z\biggr)\\
=& \ T\succ_\alpha(U\prec_\gamma W),
\end{align*}
verifying Eq.~(\mref{eq:ddf2}). Finally,
\begin{align*}
&T\succ_\alpha(U\succ_\beta W)\\
=& \ \stree x\succ_\alpha\biggl(\stree y\succ_\beta \stree z\biggr)
= \stree x\succ_\alpha\biggl(\stree y\succ_\beta(|\Ve,z,1,1;|)\biggr)\\
=& \ \stree x\succ_\alpha\Biggl(\biggl(\stree y\succ_\beta |+\stree y\prec_1|\biggr)\Ve z,\beta,1;|\Biggr)\quad(\text{by Eq.~(\mref{eq:recur})})\\
=& \ \stree x\succ_\alpha\biggl(\stree y\Ve z,\beta,1;|\biggr)\quad(\text{by Item~\mref{it:intial1} in Definition~\mref{defn:recur} and Eq.~(\mref{eq:conv}}))\\
=& \ \biggl(\stree x\succ_\alpha \stree y+\stree x\prec_\beta\stree y\biggr)\Ve z,\alpha\beta,1;|\quad(\text{by Eq.~(\mref{eq:recur})})\\
=& \ \biggl(\XX{\xxr{-5}5
\node at (-0.45,0) {$\alpha$};
\xxhu00y \xxhu{-5}5{x}
}+\XX{\xxl55
\node at (0.4,0) {$\beta$};
\xxhu00x \xxhu55y
}\biggr)\Ve z,\alpha\beta,1;|\quad(\text{by Example~\mref{exam:tu}})\\
=& \ \Biggl(\biggl(\XX{\xxr{-5}5
\node at (-0.45,0) {$\alpha$};
\xxhu00y \xxhu{-5}5{x}
}+\XX{\xxl55
\node at (0.4,0) {$\beta$};
\xxhu00x \xxhu55y
}\biggr)\succ_{\alpha\beta}|+\biggl(\XX{\xxr{-5}5
\node at (-0.45,0) {$\alpha$};
\xxhu00y \xxhu{-5}5{x}
}+\XX{\xxl55
\node at (0.4,0) {$\beta$};
\xxhu00x \xxhu55y
}\biggr)\prec_1|\Biggr)\Ve z,\alpha\beta,1;|\\
& \quad(\text{by Item~\mref{it:intial1} in Definition~\mref{defn:recur} and Eq.~(\mref{eq:conv}}))\\
=& \ \biggl(\XX{\xxr{-5}5
\node at (-0.45,0) {$\alpha$};
\xxhu00y \xxhu{-5}5{x}
}+\XX{\xxl55
\node at (0.4,0) {$\beta$};
\xxhu00x \xxhu55y
}\biggr)\succ_{\alpha\beta}(|\Ve z,1,1;|)\quad(\text{by Eq.~(\mref{eq:recur})})\\
=& \ \biggl(\stree x\succ_\alpha \stree y+\stree x\prec_\beta\stree y\biggr)\succ_{\alpha\beta}\stree z\quad(\text{by Example~\mref{exam:tu}})\\
=& \ (T\succ_\alpha U+T\prec_\beta U)\succ_{\alpha\beta} W.
\end{align*}
This completes the proof of the initial step. Assume that the conclusion holds for $\dep(T)+\dep(U)+\dep(W)=k$ and consider
the induction step of $\dep(T)+\dep(U)+\dep(W)=k+1\geq 4$. We have
\begin{align*}
 &(T\prec_\beta U)\prec_\gamma W\\
=& \ \Big(T^{l}\Ve x,\alpha_1,\alpha_2\beta;(T^{r}\prec_\beta U+T^{r}\succ_{\alpha_2} U)\Big)\prec_\gamma W\quad(\text{by Eq.~(\mref{eq:recur1}}))\\
=& \ T^{l}\Ve x,\alpha_1,\alpha_2\beta\gamma;\Big((T^{r}\prec_\beta U+T^{r}\succ_{\alpha_2} U)\prec_\gamma W+(T^{r}\prec_\beta U+T^{r}\succ_{\alpha_2} U)\succ_{\alpha_2\beta} W\Big)\\
 &\hspace{4cm} \quad(\text{by Eq.~(\mref{eq:recur1}}))\\
=& \ T^{l}\Ve x,\alpha_1,\alpha_2\beta\gamma;\Big((T^{r}\prec_\beta U)\prec_\gamma W+(T^{r}\succ_{\alpha_2} U)\prec_\gamma W
 +T^{r}\succ_{\alpha_2}(U\succ_\beta W)\Big)\\
 &\hspace{2cm} \quad(\text{by the induction hypothesis and Eq.~(\mref{eq:ddf3}}))\\
=& \ T^{l}\Ve x,\alpha_1,\alpha_2\beta\gamma;\Big(T^{r}\prec_{\beta\gamma}(U\prec_\gamma W+U\succ_\beta W)
 +T^{r}\succ_{\alpha_2}(U\prec_\gamma W)
+T^{r}\succ_{\alpha_2}(U\succ_\beta W)\Big)\\
&\hspace{2cm} \quad(\text{by the induction hypothesis and Eq.~(\mref{eq:ddf1}}))\\
=& \ T^{l}\Ve x,\alpha_1,\alpha_2\beta\gamma;\Big(T^{r}\prec_{\beta\gamma}(U\prec_\gamma W+U\succ_\beta W)+T^{r}\succ_{\alpha_2}(U\prec_\gamma W+U\succ_\beta W)\Big)\\
=& \ (T^{l}\Ve x,\alpha_1,\alpha_2; T^{r})\prec_{\beta\gamma}(U\prec_\gamma W+U\succ_\beta W)\quad(\text{by Eq.~(\mref{eq:recur1}}))\\
=& \ T\prec_{\beta\gamma}(U\prec_\gamma W+U\succ_\beta W).
\end{align*}
This proves Eq.~(\mref{eq:ddf1}) in Definition~\mref{defn:dend}.
Next, we prove Eq.~(\mref{eq:ddf2}):
\begin{align*}
 &(T\succ_\alpha U)\prec_\gamma W\\
=& \ \Big((T\succ_\alpha U^{l}+T\prec_{\beta_1} U^{l})\Ve y,\alpha\beta_1,\beta_2;U^{r}\Big)\prec_\gamma W\quad(\text{by Eq.~(\mref{eq:recur}}))\\
=& \ (T\succ_\alpha U^{l}+T\prec_{\beta_1} U^{l})\Ve y,\alpha\beta_1,\beta_2\gamma;(U^{r}\prec_\gamma W+U^{r}\succ_{\beta_2} W)\quad(\text{by Eq.~(\mref{eq:recur1}}))\\
=& \ T\succ_\alpha\Big(U^{l}\Ve y,\beta_1,\beta_2\gamma;(U^{r}\prec_\gamma W+U^{r}\succ_{\beta_2} W)\Big)\quad(\text{by Eq.~(\mref{eq:recur}}))\\
=& \ T\succ_\alpha\Big((U^{l}\Ve y,\beta_1,\beta_2;U^{r})\prec_\gamma W\Big)\quad(\text{by Eq.~(\mref{eq:recur1}}))\\
=& \ T\succ_\alpha(U\prec_\gamma W).
\end{align*}
\noindent Finally, we prove Eq.~(\mref{eq:ddf3}):
\begin{align*}
 &T\succ_\alpha(U\succ_\beta W)\\
=& \ T\succ_\alpha\Big(U\succ_\beta(W^{l}\Ve z,\gamma_1,\gamma_2;W^{r})\Big)\\
=& \ T\succ_\alpha\Big((U\succ_\beta W^{l}+U\prec_{\gamma_1} W^{l})\Ve z,\beta\gamma_1,\gamma_2;W^{r}\Big)\quad(\text{by Eq.~(\mref{eq:recur}}))\\
=& \ \Big(T\succ_\alpha(U\succ_\beta W^{l}+U\prec_{\gamma_1} W^{l})+T\prec_{\beta\gamma_1}(U\succ_\beta W^{l}+U\prec_{\gamma_1} W^{l})\Big)\Ve z,\alpha\beta\gamma_1,\gamma_2;W^{r}\quad(\text{by Eq.~(\mref{eq:recur}}))\\
=& \ \Big(T\succ_\alpha(U\succ_\beta W^{l})+T\succ_\alpha(U\prec_{\gamma_1} W^{l})+T\prec_{\beta\gamma_1}(U\succ_\beta W^{l}+U\prec_{\gamma_1} W^{l})\Big)\Ve z,\alpha\beta\gamma_1,\gamma_2;W^{r}\\
=& \ \Big((T\succ_\alpha U+T\prec_\beta U)\succ_{\alpha\beta} W^{l}+(T\succ_\alpha U)\prec_{\gamma_1} W^{l}+(T\prec_{\beta}U)\prec_{\gamma_1} W^{l}\Big)\Ve z,\alpha\beta\gamma_1,\gamma_2;W^{r}\\
&\hspace{3cm}(\text{by the induction hypothesis and Eqs.~(\mref{eq:ddf1})-(\mref{eq:ddf3})})\\
=& \ \Big((T\succ_\alpha U+T\prec_\beta U)\succ_{\alpha\beta} W^{l}+(T\succ_\alpha U+T\prec_\beta U)\prec_{\gamma_1} W^{l}\Big)\Ve z,\alpha\beta\gamma_1,\gamma_2;W^{r}\\
=& \ (T\succ_\alpha U+T\prec_\beta U)\succ_{\alpha\beta}(W^{l}\Ve z,\gamma_1,\gamma_2; W^{r})\quad(\text{by Eq.~(\mref{eq:recur}}))\\
=& \ (T\succ_\alpha U+T\prec_\beta U)\succ_{\alpha\beta}W.
\end{align*}
This completes the proof of step 1.

{\noindent \bf Step 2:} We show that $\DDF$ is generated by $\biggl\{\stree x\Bigm| x\in X\biggr\}.$ Let $T=T^{l}\Ve x,\alpha_1,\alpha_2;T^{r}$ with $x\in X$ and $\alpha_1,\alpha_2\in\Omega^{1}$. We employ induction on $\dep(T)\geq 1$.
For the initial step of $\dep(T)=1$, we have
$$T=|\Ve x,1,1;|=\stree x\in\biggl\{\stree x\Bigm| x\in X\biggr\}.$$
Assume that the conclusion holds for $T$ with $\dep(T)=k$ and consider the induction step of $\dep(T)=k+1\geq 2$.
Since $T^l$ and $T^r$ can not be $|$ simultaneously, there exist three cases to consider.

{\noindent \bf Case 1:} $T^{l}=|$ and $T^{r}\neq |$. Then $\alpha_1=1$ and
\begin{align*}
T=& \ T^{l}\Ve x,\alpha_1,\alpha_2;T^{r}=|\Ve x,1,\alpha_2;T^{r}\\
=& \  |\Ve x,1,\alpha_2;(|\prec_{\alpha_2}T^{r}+|\succ_1T^{r})\quad(\text{by Item~\mref{it:intial1} in Definition~\mref{defn:recur} and Eq.~(\mref{eq:conv})})\\
=& \ (|\Ve x,1,1;|)\prec_{\alpha_2}T^{r}\quad(\text{by Eq.~(\mref{eq:recur1})})\\
=& \ \stree x\prec_{\alpha_2}T^{r}.
\end{align*}
By the induction hypothesis, $T^{r}$ is generated by $\biggl\{\stree x\Bigm| x\in X\biggr\}$ and so $T$ is generated by $\biggl\{\stree x\Bigm| x\in X\biggr\}.$

{\noindent \bf Case 2:} $T^{l}\neq|$ and $T^{r}=|$. This case is similar to Case $1$.

{\noindent \bf Case 3:} $T^{l}\neq |$ and $T^{r}\neq|$. Then
\begin{align*}
T=& \ T^{l}\vee_{x,\,(\alpha_1,\,\alpha_2)}T^{r}\\
=& \ T^{l}\vee_{x,\,(\alpha_1,\,\alpha_2)}(|\prec_{\alpha_2}T^{r}+|\succ_1T^{r})\quad(\text{by Item~\mref{it:intial1} in Definition~\mref{defn:recur} and Eq.~(\mref{eq:conv})})\\
=& \ (T^{l}\vee_{x,\,(\alpha_1,\,1)}|)\prec_{\alpha_2}T^{r}\quad(\text{by Eq.~(\mref{eq:recur1})})\\
=& \  \Big((T^{l}\succ_{\alpha_1}|+T^{l}\prec_1|)\vee_{x,\,(\alpha_1,\,1)}|\Big)\prec_{\alpha_2}T^{r}\quad(\text{by Item~\mref{it:intial1} in Definition~\mref{defn:recur} and Eq.~(\mref{eq:conv})})\\
=& \ \Big(T^{l}\succ_{\alpha_1}(|\vee_{x,\,(1,\,1)}|)\Big)\prec_{\alpha_2}T^{r}\quad(\text{by Eq.~(\mref{eq:recur})})\\
= & \ \biggl(T^{l}\succ_{\alpha_1}\stree x\biggr)\prec_{\alpha_2}T^{r}.
\end{align*}
By the induction hypothesis, $T^{l}$ and $T^{r}$ are generated by $\biggl\{\stree x\Bigm| x\in X\biggr\}$ and so $T$ is generated by $\biggl\{\stree x\Bigm| x\in X\biggr\}.$ This completes the proof.
\end{proof}

\noindent Now we arrive at our main result in this section.
\begin{defn}
 Let $X$ be a set and let $\Omega$ be a semigroup.
 A free dendriform family algebra on $X$ is a dendriform family algebra $(D, \{\prec_{\omega},\succ_{\omega}\mid \omega\in \Omega\})$ together with a map $j: X\rightarrow D$ that satisfies the following universal property:
 for any dendriform family algebra $(D', \{\prec'_\omega,\succ'_\omega\mid \omega\in \Omega\})$
and map $f: X\rightarrow D',$ there is a unique dendriform family algebra morphism $\bar{f}:
D\rightarrow D'$ such that $f=\bar{f}\circ j.$ The free dendriform family algebra on $X$ is unique up to isomorphism.
\end{defn}

\noindent Let $j: X \ra \DDF$ be the map defined by $j(x)=\stree x$ for $x\in X$.
\begin{theorem}
Let $X$ be a set and let $\Omega$ be a semigroup. 
Then $(\DDF,\{\prec_\omega,\succ_\omega\mid\omega\in\Omega\})$, together with the map $j$, is the free dendriform family algebra on $X$.
\mlabel{thm:free1}
\end{theorem}

\begin{proof}
By Proposition~\mref{prop:dend}, we are left to show that $(\DDF,\{\prec_\omega,\succ_\omega\mid\omega\in\Omega\})$ satisfies the universal property. For this, let $(D,\{\prec'_{\omega},\succ'_{\omega}\mid\omega\in\Omega\})$ be a dendriform family algebra. First of all, if there exists a dendriform family algebra morphism $\bar{f}: \DDF\to D$ extending $f:X\to D$ in the sense that $\bar f(\stree{x})=f(x)$ for any $x\in X$, then such an $\bar f$ is unique, due to the fact that the collection of trees $\left(\stree{x}\right)_{x\in X}$ generate the dendriform family algebra $\DDF$.\\

\noindent Now let us define a linear map $$\bar{f}: \DDF\rightarrow D,\quad T\mapsto\bar{f}(T)$$ by induction on $\dep(T)\geq 1$.
Let us write $T=T^{l}\Ve x,\alpha_1,\alpha_2;T^{r}$ with $x\in X$ and $\alpha_1,\alpha_2\in\Omega^{1}$. For the initial step $\dep(T)=1$, we have $T=\stree x$ for some $x\in X$ and define
\begin{equation}
\bar{f}(T):=f(x).
\mlabel{eq:init}
\end{equation}
Assume that $\bar{f}(T)$ has been defined for $T$ with $\dep(T)=k$ and consider the induction step of $\dep(T)=k+1\geq 2$.
Note that $T^l$ and $T^r$ can not be $|$ simultaneously and define
\begin{align}
\bar{f}(T):=& \  \bar{f}(T^{l}\Ve x,\alpha_1,\alpha_2;T^{r})\nonumber\\
:=& \ \left\{
\begin{array}{ll}
f(x)\prec'_{\alpha_2}\bar{f}(T^{r}), & \ \text{ if } T^{l}=|\neq T^{r};\\
\bar{f}(T^{l})\succ'_{\alpha_1} f(x), & \ \text{ if } T^{l}\neq |=T^{r};\\
\bigl(\bar{f}(T^{l})\succ'_{\alpha_1}f(x)\bigr)\prec'_{\alpha_2}\bar{f}(T^{r}), & \ \text{ if }T^{l}\neq |\neq T^{r}.
\mlabel{eq:induc}
\end{array}
\right .
\end{align}

\noindent We are left to prove that $\bar{f}$ is a morphism of dendriform family algebras:
 \begin{equation*}
\bar{f}(T\prec_\omega U)=\bar{f}(T)\prec'_\omega\bar{f}(U)\,\text{ and }\,
\bar{f}(T\succ_\omega U)=\bar{f}(T)\succ'_\omega\bar{f}(U),
\mlabel{eq:comff}
\end{equation*}
in which we only prove the first equation by induction on $\dep(T)+\dep(U)\geq 2$,
as the proof of the second one is similar.
Write
$$T=T^{l}\Ve x,\alpha_1,\alpha_2;T^{r}\,\text{ and }\, U=U^{l}\Ve y,\beta_1,\beta_2;U^{r}.$$
For the initial step $\dep(T)+\dep(U)=2$, we have $T=\stree x$ and $U=\stree y$ for some $x,y\in X$.
So
\begin{align*}
\bar{f}(T\prec_\omega U)=& \ \bar{f}\biggl((|\Ve x,1,1;|)\prec_\omega \stree y\biggr)\\
=& \ \bar{f}\Biggl(|\Ve x,1,\omega;\biggl(|\prec_\omega \stree y+|\succ_1\stree y\biggr)\Biggr)\quad(\text{by Eq.~(\mref{eq:recur1})})\\
=& \ \bar{f}\biggl(|\Ve x,1,\omega;\stree y\biggr)\quad(\text{by Item~\mref{it:intial1} in Definition~\mref{defn:recur} and Eq.~(\mref{eq:conv})})\\
=& \ f(x)\prec'_\omega \bar{f}\biggl(\stree y\biggr)\quad(\text{by Eq.~(\mref{eq:induc})})\\
=& \ \bar{f}\biggl(\stree x\biggr)\prec'_\omega \bar{f}\biggl(\stree y\biggr)\quad(\text{by Eq.~(\mref{eq:init})})\\
=& \ \bar{f}(T)\prec'_\omega\bar{f}(U).
\end{align*}
Assume that the conclusion holds for $T$ and $U$ in $\DDF$ with $\dep(T)+\dep(U)=k$ and consider the case when $\dep(T)+\dep(U)=k+1\geq 3,$
there are four cases to consider.

{\noindent \bf Case 1:} $T^{l}=|$ and $T^{r}\neq |$. Then $\alpha_1=1$ and
\begin{align*}
\bar{f}(T\prec_\omega U)=& \ \bar{f}\bigl((|\Ve x,1,\alpha_2;T^{r})\prec_\omega U\bigr)\\
=& \ \bar{f}\bigl(|\Ve x,1,\alpha_2\omega;(T^{r}\prec_\omega U+T^{r}\succ_{\alpha_2}U)\bigr)\quad(\text{by Eq.~(\mref{eq:recur1})})\\
=& \ f(x)\prec'_{\alpha_2\omega}\bar{f}(T^{r}\prec_\omega U+T^{r}\succ_{\alpha_2} U)\quad(\text{by Eq.~(\mref{eq:induc})})\\
=& \ f(x)\prec'_{\alpha_2\omega}\bigl(\bar{f}(T^{r})\prec'_\omega \bar{f}(U)+\bar{f}(T^{r})\succ'_{\alpha_2}\bar{f}(U)\bigr)\\
& \ (\text{by the induction hypothesis on $\dep(T)+\dep(U)$})\\
=& \ \bigl(f(x)\prec'_{\alpha_2}\bar{f}(T^{r})\bigr)\prec'_\omega\bar{f}(U)\quad(\text{by Eq.~(\mref{eq:ddf1})})\\
=& \ \bar{f}(|\Ve x,1,\alpha_2;T^{r})\prec'_\omega \bar{f}(U)\quad(\text{by Eq.~(\mref{eq:induc})})\\
=& \ \bar{f}(T)\prec'_\omega\bar{f}(U).
\end{align*}
{\noindent \bf Case 2:} $T^{l}\neq|$ and $T^{r}=|$. This case is similar to Case 1.

{\noindent \bf Case 3:} $T^{l}\neq|$ and $T^{r}\neq |$. Then
\begin{align*}
\bar{f}(T\prec_\omega U)=& \ \bar{f}\bigl((T^{l}\Ve x,\alpha_1,\alpha_2;T^{r})\prec_\omega U\bigr)\\
=& \ \bar{f}\bigl(T^{l}\Ve x,\alpha_1,\alpha_2\omega;(T^{r}\succ_{\alpha_2} U+T^{r}\prec_\omega U)\bigr)\quad(\text{by Eq.~(\mref{eq:recur1})})\\
=& \ \bigl(\bar{f}(T^{l})\succ'_{\alpha_1} f(x)\bigr)\prec'_{\alpha_2\omega}\bar{f}(T^{r}\succ_{\alpha_2} U+T^{r}\prec_\omega U)\quad(\text{by Eq.~(\mref{eq:induc})})\\
=& \ \bigl(\bar{f}(T^{l})\succ'_{\alpha_1} f(x)\bigr)\prec'_{\alpha_2\omega}\bigl(\bar{f}(T^{r})\succ'_{\alpha_2} \bar{f}(U)+\bar{f}(T^{r})\prec'_\omega\bar{f}(U)\bigr)\\
&\hspace{0.5cm}(\text{by the induction hypothesis on $\dep(T)+\dep(U)$})\\
=& \ \Bigl(\bigl(\bar{f}(T^{l})\succ'_{\alpha_1} f(x)\bigr)\prec'_{\alpha_2} \bar{f}(T^{r})\Bigr)\prec'_\omega\bar{f}(U)\quad(\text{by Eq.~(\mref{eq:ddf1})})\\
=& \ \bar{f}(T^{l}\Ve x,\alpha_1,\alpha_2; T^{r})\prec'_\omega \bar{f}(U)\quad(\text{by Eq.~(\mref{eq:induc})})\\
=& \ \bar{f}(T)\prec'_\omega\bar{f}(U).
\end{align*}

{\noindent \bf Case 4:} $T^{l} = |$ and $T^{r} = |$. Then $T = \stree x$ for some $x\in X$ and we have
\begin{align*}
\bar f(T\prec_\omega U)&=\bar f\left(|\vee_{x,(1,\omega)}U\right)\\
&=f(x)\prec'_\omega\bar f(U)\\
&= \bar f(T)\prec'_\omega\bar f(U)
\end{align*}
according to \eqref{eq:induc}.

\noindent This completes the proof.
\end{proof}

\section{Free tridendriform family algebras}
\mlabel{sec:tri}
In this section, we construct free tridendriform family algebras in terms of typed valently decorated Schr\"oder trees.
For this, let us first recall the construction of free tridendriform algebras. See~\cite{LoRo98,LoRo04} for more details.

\subsection{Free tridendriform algebras}\mlabel{sub:trid}
Let $X$ be a set. For $n\geq 0,$
let $T_{n,\,X}$ be the set of planar rooted trees with $n+1$ leaves
and with vertices valently decorated by elements of $X$, in the sense that if a vertex has valence $k$, then the vertex is decorated by an element in $X^{k-1}$. For example, the vertex of $\stree x$ is decorated by $x\in X$ while the vertex of
$\XX{\xx002 \xxh0023{x\ }{0.5} \xxh0012{\ \,y}{0.4}}$ is decorated by $(x,y)\in X^{2}$.
Here are the first few of them.
\begin{align*}
T_{0,\,X}=& \ \{|\},\qquad T_{1,\,X}=\left\{\stree x\Bigm|x\in X\right\},
\qquad T_{2,\,X}=\left\{
\XX{\xxr{-5}5
\xxhu00x \xxhu{-5}5y
}, \,
\XX{\xxl55
\xxhu00x \xxhu55y
}, \,
\XX{\xx002
\xxh0023{x\ }{0.5} \xxh0012{\ \,y}{0.4}
}\Bigm|x,y\in X
\right\},\\
T_{3,\,X}=& \ \left\{
\XX[scale=1.6]{\xxr{-4}4\xxr{-7.5}{7.5}
\xxhu00x \xxhu[0.1]{-4}4{\,y} \xxhu[0.1]{-7.5}{7.5}{z}
},
\XX[scale=1.6]{\xxr{-5}5\xxl{-2}8
\xxhu00x
\xxhu[0.1]{-5}5{y} \xxhu[0.1]{-2}8{z}
},
\XX[scale=1.6]{\xxr{-4}4\xx{-4}42
\xxhu00x \xxh{-4}423{y\ \,}{0.3} \xxh{-4}412{\ \, z}{0.3}
},
\XX[scale=1.6]{\xx00{1.6}\xx00{2.4}
\xxh001{1.6}{\ \ \,z}{0.6}
\xxh00{1.6}{2.4}{y}{0.5}
\xxh00{2.4}3{x\ \ }{0.6}
},
\XX[scale=1.6]{\xxr{-6}6\xxl66
\xxhu00{x} \xxhu66{z} \xxhu[0.1]{-6}6{y}
},
\XX[scale=1.6]{\xxlr0{7.5} \draw(0,0)--(0,0.75);
\xxh0023{x\ \,}{0.3} \xxhu[0.12]0{7.5}{z}
\xxh0012{\ \ y}{0.22}
},\ldots\Bigg|\,x,y,z\in X
\right\}.
\end{align*}

For $T^{(i)}\in T_{n_i,\,X}, 0\leq i\leq k$, and $x_1,\ldots,x_k\in X,$ the grafting $\bigvee$ of $T^{(i)}$ over $(x_1,\ldots,x_k)$ is
$$\bigvee\nolimits_{x_1,\ldots,x_k}^{k+1}(T^{(0)},\ldots,T^{(k)}).$$
Any planar rooted tree $T$ can be uniquely expressed as such a grafting of lower depth trees.
For example,
$$\XX{\xx002 \xxh0023{x\ }{0.5} \xxh0012{\ \,y}{0.4}}=\bigvee\nolimits_{x,y}\,(|,\,|,\,|).$$

Let $\mathrm{DT}(X):=\underset{n\geq 1}\bigoplus\,\bfk T_{n,\,X}.$ Define binary operations $\prec,\succ$ and $\cdot$ on $\mathrm{DT}(X)$ recursively on $\dep(T)+\dep(U)$ by
\begin{enumerate}
\item $|\succ T:=T\prec |:=T,\, |\prec T:=T\succ |:=0 \text{ and }\, |\cdot T:=T\cdot |:=0$ for $T\in T_{n,\,X}$ with $n\geq 1.$
\item For $\bigvee_{x_1,\ldots,x_m}^{m+1}(T^{(0)},\ldots,T^{(m)})$ and $\bigvee_{y_1,\ldots,y_n}^{n+1}(U^{(0)},\ldots,U^{(n)}),$ define
\begin{align*}
T\prec U:=& \ \bigvee\nolimits_{x_1,\ldots,x_{m-1},x_m}^{m+1}(T^{(0)},\ldots,T^{(m-1)},T^{(m)}\succ U+T^{(m)}\prec U+T^{(m)}\cdot U),\\
T\succ U:=& \ \bigvee\nolimits_{y_1,y_2,\ldots,y_n}^{n+1}(T\succ U^{(0)}+T\prec U^{(0)}+T\cdot U^{(0)},U^{(1)},\ldots,U^{(n)}),\\
T\cdot U:=& \ \bigvee\nolimits_{x_1,\ldots,x_{m-1},x_m,\,y_1,\ldots,y_n}^{m+n+1}(
T^{(0)},\ldots,T^{(m-1)},T^{(m)}\succ U^{(0)}+T^{(m)}\prec U^{(0)}+T^{(m)}\cdot U^{(0)},U^{(1)},\ldots,U^{(n)}).
\end{align*}
\end{enumerate}

Let $j:X\ra\mathrm{DT}(X)$ be the map defined by $j(x)=\stree x$ for any $x\in X$. The following result is well-known.
\begin{theorem}\cite{LoRo98,LoRo04}
Let $X$ be a set. Then $(\mathrm{DT}(X),\prec,\succ,\cdot)$, together with $j$, is the free tridendriform algebra on $X.$
\end{theorem}

\subsection{Free tridendriform family algebras}
In this subsection, we introduce typed valently decorated Schr\"oder trees
to construct free tridendriform family algebras. Let us first propose

\begin{defn}
Let $X$  and $\Omega$ be two sets. A {\bf $X$-valently decorated $\Omega$-typed (abbreviated typed valently decorated) Schr\"oder tree}
is a triple $T = (T,\dec, \type)$, where
\begin{enumerate}
\item $T$ is a Schr\"oder tree.
\item $\dec: IV(T)\ra \sqcup_{n\geq 1} X^n$ is a map sending any vertex of arity $n+1$ to $X^n$.
\item $\type: IE(T)\ra \Omega$ is a map.
\end{enumerate}
\mlabel{defn:tdtree1}
\end{defn}

\begin{remark}
The trees in $\mathrm{DT(X)}$ employed to construct free tridendriform algebras
are indeed $X$-valently decorated \{$\star$\}-typed Schr\"oder trees, where \{$\star$\} is any set with one single element.
\end{remark}

Let $X$ be a set and let $\Omega$ be a semigroup.
For $n\geq 0,$ let $T_n:=T_{n,\,X,\,\Omega}$ be the set of $X$-valently
decorated $\Omega^{1}$-typed Schr\"oder trees with $n+1$ leaves, such that
the leaves are decorated by the identity 1 in $\Omega^{1}$ and internal edges are decorated by elements of $\Omega$.
Let us expose some examples. Note that the edge decoration 1 is omitted for convenience.
\begin{align*}
T_0=& \ \{|\},\qquad
T_1=\left\{\stree x\Bigm|x\in X\right\},
\qquad T_2=\left\{
\XX{\xxr{-5}5
\node at (-0.4,0) {$\alpha$};
\xxhu00x \xxhu{-5}5y
},
\XX{\xxl55
\node at (0.4,0) {$\alpha$};
\xxhu00x \xxhu55y
},
\XX{\xx002
\xxh0023{x\ }{0.5} \xxh0012{\ \,y}{0.4}
}\Bigm|x,y\in X,\alpha\in\Omega
\right\},\\
T_3=& \ \left\{
\XX[scale=1.6]{\xxr{-4}4\xxr{-7.5}{7.5}
\node at (-0.3,0) {$\alpha$};
\node at (-0.72,0.4) {$\beta$};
\xxhu00x \xxhu[0.1]{-4}4{\,y} \xxhu[0.1]{-7.5}{7.5}{z}
},
\XX[scale=1.6]{\xxr{-5}5\xxl{-2}8
\node at (-0.4,0.1) {$\alpha$};
\node at (-0.2,0.55) {\tiny $\beta$};
\xxhu00x
\xxhu[0.1]{-5}5{y} \xxhu[0.1]{-2}8{z}
},
\XX[scale=1.6]{\xxr{-4}4\xx{-4}42
\node at (-0.3,0) {$\alpha$};
\xxhu00x \xxh{-4}423{y\ \,}{0.3} \xxh{-4}412{\ \, z}{0.3}
},
\XX[scale=1.6]{\xx00{1.6}\xx00{2.4}
\xxh001{1.6}{\ \ \,z}{0.6}
\xxh00{1.6}{2.4}{y}{0.5}
\xxh00{2.4}3{x\ \ }{0.6}
},
\XX[scale=1.6]{\xxlr0{7.5} \draw(0,0)--(0,0.75);
\xxh0023{x\ \,}{0.3} \xxhu[0.12]0{7.5}{z}
\node at (0.16,0.35) {\tiny $y$};
\node at (0.17,0.62) {\tiny $\alpha$};
},
\ldots\Bigg|\,x,y,z\in X,\alpha,\beta\in\Omega
\right\}.
\end{align*}
Denote by $$\DTF:=\bigoplus_{n\geq 1}\bfk T_{n}.$$
For typed valently decorated Schr\"oder trees $T^{(i)}\in T_{n_i}, 0\leq i\leq k,$ and $x_1,\ldots,x_k\in X, \alpha_0,\ldots,\alpha_k\in\Omega^{1},$ the grafting $\bigvee$ of $T^{(i)}$ over $(x_1,\ldots,x_k)$ and $(\alpha_0,\ldots,\alpha_k)$ is
\begin{equation}
T=\bigv x_1,\dots,x_k;k+1;\alpha_0,\dots,\alpha_k;
(T^{(0)},\dots,T^{(k)})
\mlabel{eq:texpr}
\end{equation}
obtained by joining the $k+1$ roots of $T^{(i)}$ to a new vertex valently decorated by $(x_1,\ldots,x_k)$ and
decorating the new edges from left to right by $\alpha_0,\ldots,\alpha_k$ respectively. 

\begin{exam}
Let $T^{(0)}=\stree x, T^{(1)}=\stree y$, $T^{(2)}=\stree z,$ $u,v\in X$ and $\alpha,\beta,\gamma\in\Omega.$ Then
$$
\bigv u,v;3;\alpha,\beta,\gamma;\biggl(\stree x,\stree y,\stree z\biggr)
=\XX[scale=1.6]{
\xxlr0{7.5}\xxl77 \xxr{-7}7\draw(0,0)--(0,0.75);
\xxhu77{z}\xxhu{-7}7{x}\xxhu[0.1]0{7.5}{y}
\xxh0023{u\ \,}{0.3}
\node at (0.14,0.35) {$v$};
\node at (0.18,0.65) {\tiny $\beta$};
\node at (-0.6,0.25) {$\alpha$};
\node at (0.6,0.25) {$\gamma$};
}.$$
\end{exam}

Any typed valently decorated Schr\"oder tree $T\in \DTF$ can be uniquely expressed as such a grafting of lower depth typed valently decorated Schr\"oder trees in Eq.~(\mref{eq:texpr}),
and we define the breadth of $T$ as the arity of the internal vertex adjacent to the root, i.e. $\bre(T) := k+1$. For example,
\begin{align*}\XX{\xx002
\xxh0023{x\ }{0.5} \xxh0012{\ \,y}{0.4}
}=&\, \bigv x,y;3;1,1,1;(|,|,|);,\quad  \bre\biggl(\XX{\xx002
\xxh0023{x\ }{0.5} \xxh0012{\ \,y}{0.4}
}\biggr) = 3,\, \text{ and }\,\\
T =& \XX[scale=1.6]{\xxlr0{7.5}\xxl77 \draw(0,0)--(0,0.75);
\xxh0023{x\ \,}{0.3} \xxhu[0.1]0{7.5}{y} \xxhu77{u}
\node at (0.15,0.38) {$z$};
\node at (0.17,0.65) {$\alpha$};
\node at (0.55,0.25) {$\beta$};
}=\bigv x,z;3;1,\alpha,\beta;\biggl(|,\stree y,\stree u\biggr), \quad \bre(T) = 3.
\end{align*}

\begin{defn}
Let $X$ be a set and let $\Omega$ be a semigroup.
 Define binary operations $\{\prec_\omega,\succ_\omega\mid\omega\in\Omega\}$ and $\cdot$ on $\DTF$ recursively on $\dep(T)+\dep(U)$ by
\begin{enumerate}
\item $|\succ_\omega T:=T\prec_\omega |:=T,\, |\prec_\omega T:=T\succ_\omega |:=0 \text{ and }\, |\cdot T:=T\cdot |:=0$ for $\omega\in \Omega$ and $T\in T_{n}$ with $n\geq 1.$\mlabel{it:trida}

\item Let
\begin{align*}
T=& \ \bigv x_1,\ldots,x_m;m+1;\alpha_0,\ldots,\alpha_m;(T^{(0)},\ldots, T^{(m)})\in T_{m},\\
U=& \ \bigv y_1,\ldots,y_n;n+1;\beta_0,\ldots,\beta_n;(U^{(0)},\ldots, U^{(n)})\in T_{n},
\end{align*}
define
\begin{align}
 T\prec_\omega U:=& \ \bigv x_1,\ldots,x_{m-1},x_m;m+1;\alpha_0,\ldots,\alpha_{m-1},\alpha_m\omega; (T^{(0)},\ldots,T^{(m-1)},T^{(m)}\succ_{\alpha_m} U+T^{(m)}\prec_\omega U+T^{(m)}\cdot U),\mlabel{eq:tdpre}\\
T\succ_\omega U:=& \ \bigv y_1,y_2,\ldots,y_n;n+1;\omega\beta_0,\beta_1,\ldots,\beta_n;  (T\succ_\omega U^{(0)}+T\prec_{\beta_0} U^{(0)}+T\cdot U^{(0)},U^{(1)},\ldots,U^{(n)}),\mlabel{eq:tdsuc}\\
 T\cdot U:=& \ \bigv x_1,\ldots,x_{m-1},x_m,y_1,\ldots,y_n;m+n+1;\alpha_0,\ldots,\alpha_{m-1},\alpha_m\beta_0,\beta_1,
 \ldots,\beta_n;
 (T^{(0)},\ldots,T^{(m-1)},T^{(m)}\succ_{\alpha_m} U^{(0)}+T^{(m)}\prec_{\beta_0} U^{(0)}+T^{(m)}\cdot U^{(0)},\mlabel{eq:tdcdot}\\
 & \ U^{(1)},\ldots,U^{(n)}).\nonumber
\end{align}
\end{enumerate}
\mlabel{defn:trirec}
\end{defn}
\noindent Here in Eq.~(\mref{eq:tdcdot}) when $T^{(m)} = | = U^{(0)}$, we employ the convention that
\begin{equation}
|\prec_1 |+|\succ_1 |+|\cdot | :=|,
\mlabel{eq:1star1}
\end{equation}
and
\begin{equation}
|\succ_1 T:=T\prec_1 |:=T\,\text{ and }\, |\prec_1 T:=T\succ_1 |:=0.
\mlabel{eq:con}
\end{equation}

\noindent Let us expose some examples for better understanding.

\begin{exam}
Let $T=\stree x$ and $U=\stree y$ with $x,y\in X$. For $\omega\in\Omega$,
\begin{align*}
\stree x\prec_\omega \stree y
=& \ \bigv x;2;1,1;(|,|)\prec_\omega \stree y\\
=& \ \bigv x;2;1,\omega;\biggl(|,|\prec_\omega \stree y+|\succ_1\stree y+|\cdot\stree y\biggr)\quad(\text{by Eq.~(\mref{eq:tdpre})})\\
=& \ \bigv x;2;1,\omega;\biggl(|,\stree y\biggr)\quad(\text{by Item~\mref{it:trida} in Definition~\mref{defn:trirec} and Eq.~(\mref{eq:con})})                 \\
=& \ \XX{\xxl55
\node at (0.4,0) {$\omega$};
\xxhu00x \xxhu55y
},\\
\stree x\succ_\omega \stree y
=& \ \stree x\succ_\omega \bigv y;2;1,1;(|,|)\\
=& \ \bigv y;2;\omega,1;\biggl(\stree x\prec_1 |+\stree x\succ_\omega |+\stree x\cdot |,|\biggr)\quad(\text{by Eq.~(\mref{eq:tdsuc})})\\
=& \ \bigv y;2;\omega,1;\biggl(\stree x,|\biggr)\quad(\text{by Item~\mref{it:trida} in Definition~\mref{defn:trirec} and Eq.~(\mref{eq:con})})                 \\\\
=& \ \XX{\xxr{-5}5
\node at (-0.45,0) {$\omega$};
\xxhu00y \xxhu{-5}5{x}
},\\
\stree x\cdot \stree y=& \ \bigv x;2;1,1;(|,|)\cdot\bigv y;2;1,1;(|,|)\\
=& \ \bigv x,y;3;1,1,1;(|,|\prec_1 |+|\succ_1|+|\cdot |,|)\quad(\text{by Eq.~(\mref{eq:tdcdot})})\\
=& \ \bigv x,y;3;1,1,1;(|,|,|)\quad(\text{by Eq.~(\mref{eq:1star1})})\\
=& \ \XX{\xx002
\xxh0023{x\ }{0.5} \xxh0012{\ \,y}{0.4}
}.
\end{align*}
\mlabel{exam:tuu}
\end{exam}

\begin{exam}
Let $T=\XX{\xxr{-5}5
\node at (-0.4,0) {$\alpha$};
\xxhu00x \xxhu{-5}5y
}, U=\stree z$ with $x,y,z\in X$ and $\alpha\in\Omega$. For $\beta\in\Omega$,
\begin{align*}
T\succ_\beta U=& \ \XX{\xxr{-5}5
\node at (-0.4,0) {$\alpha$};
\xxhu00x \xxhu{-5}5y
}\succ_\beta \bigv z;2;1,1;(|,|)
=\bigv z;2;\beta,1;\biggl(\XX{\xxr{-5}5
\node at (-0.4,0) {$\alpha$};
\xxhu00x \xxhu{-5}5y
}\succ_\beta |+\XX{\xxr{-5}5
\node at (-0.4,0) {$\alpha$};
\xxhu00x \xxhu{-5}5y
}\prec_1 |+\XX{\xxr{-5}5
\node at (-0.4,0) {$\alpha$};
\xxhu00x \xxhu{-5}5y
}\cdot |,|\biggr)\\
=& \ \bigv z;2;\beta,1;\biggl(\XX{\xxr{-5}5
\node at (-0.4,0) {$\alpha$};
\xxhu00x \xxhu{-5}5y
},|\biggr)=\XX[scale=1.6]{\xxr{-4}4\xxr{-7.5}{7.5}
\node at (-0.35,0) {$\beta$};
\node at (-0.7,0.45) {$\alpha$};
\xxhu00{z} \xxhu[0.1]{-4}4{x} \xxhu[0.1]{-7.5}{7.5}{y}
}.\\
T\prec_\beta U=& \ \bigv x;2;\alpha,1;\biggl(\stree y,|\biggr)\prec_\beta \stree z=\bigv x;2;\alpha,\beta;\biggl(\stree y,|\prec_\beta \stree z+|\succ_1 \stree z+|\cdot \stree z\biggr)\\
=& \ \bigv x;2;\alpha,\beta;\biggl(\stree y,\stree z\biggr)=\XX[scale=1.6]{\xxr{-6}6\xxl66
\node at (-0.5,0.15) {$\alpha$};
\node at (0.5,0.15) {$\beta$};
\xxhu00{x} \xxhu66{z} \xxhu[0.1]{-6}6{y}
}.\\
U\cdot T=& \ \bigv z;2;1,1;(|,|)\cdot \bigv x;2;\alpha,1;\biggl(\stree y,|\biggr)
=\bigv z,x;3;1,\alpha,1;\biggl(|,|\succ_1\stree y+|\prec_\alpha \stree y+|\cdot \stree y,|\biggr)\\
=& \ \bigv z,x;3;1,\alpha,1;\biggl(|,\stree y,|\biggr)=\XX[scale=1.6]{
\xxlr0{7.5} \draw(0,0)--(0,0.75);
\xxh0023{z\ \,}{0.3} \xxhu[0.12]0{7.5}{y}
\node at (0.16,0.38) {$x$};
\node at (0.17,0.64) {$\alpha$};
}.
\end{align*}
\end{exam}

\begin{prop}
Let $X$ be a set and let $\Omega$ be a semigroup.
Then $(\DTF,\,\{\prec_\omega,\succ_\omega\mid \omega\in \Omega\},\cdot)$ is a tridendriform family algebra, generated by  $\biggl\{\stree x\Bigm|x\in X\biggr\}.$
\mlabel{prop:trid}
\end{prop}

\begin{proof}
We divide into two steps to prove the result.

{\noindent \bf Step 1:} we first prove $(\DTF, \{\prec_\omega,\succ_\omega\mid \omega\in \Omega\},\cdot)$ is a tridendriform family algebra. Let
\begin{align*}
T=& \ \bigv x_1,\ldots,x_m;m+1;\alpha_0,\ldots,\alpha_m;(T^{(0)},\ldots,T^{(m)}),\\
U=& \ \bigv y_1,\ldots,y_n;n+1;\beta_0,\ldots,\beta_n;(U^{(0)},\ldots,U^{(n)})\text{ and }\\
W=& \ \bigv z_1,\ldots,z_\ell;\ell+1;\gamma_0,\ldots,\gamma_\ell;(W^{(0)},\ldots,W^{(\ell)})
\end{align*}
be in $\DTF.$ Then we apply induction on $\dep(T)+\dep(U)+\dep(W)\geq 3.$ For the initial step of $\dep(T)+\dep(U)+\dep(W)=3$,
we have $T=\ssstree{x_1}{x_m}, U=\ssstree{y_1}{y_n}$ and $W=\ssstree{z_1}{z_\ell}$ for some $x_1,\ldots x_m,y_1,\ldots,y_n,z_1,\ldots,z_\ell\in X$. Then
{\small
\begin{align*}
&(T\prec_\beta U)\prec_\gamma W\\
=& \ \Biggl(\ssstree{x_1}{x_m}\prec_\beta\ssstree{y_1}{y_n}\Biggr)\prec_\gamma \ssstree{z_1}{z_\ell}\\
=& \ \Biggl(\bigv x_1,\ldots,x_m;m+1;1,\ldots,1;(|,\ldots,|)\prec_\beta\ssstree{y_1}{y_n}\Biggr)\prec_\gamma \ssstree{z_1}{z_\ell}\\
=& \ \Biggl(\bigv x_1,\ldots,x_{m-1},x_m;m+1;1,\ldots,1,\beta;\Biggl(|,\ldots,|,|\prec_\beta
\ssstree{y_1}{y_n}+|\succ_1\ssstree{y_1}{y_n}
+|\cdot\ssstree{y_1}{y_n}\Biggr)\Biggr)\prec_\gamma \ssstree{z_1}{z_\ell}\\
&\hspace{5cm}(\text{by Eq.~(\mref{eq:tdpre})})\\
=& \ \bigv x_1,\ldots,x_{m-1},x_m;m+1;1,\ldots,1,\beta;\Biggl(|,\ldots,|,
\ssstree{y_1}{y_n}\Biggr)\prec_\gamma\ssstree{z_1}{z_\ell}\\
&\quad(\text{by Item~\mref{it:trida} in Definition~\mref{defn:trirec} and Eq.~(\mref{eq:con})})\\
=& \ \bigv x_1,\ldots,x_{m-1},x_m;m+1;1,\ldots,1,\beta\gamma;\Biggl(|,\ldots,|,\ssstree{y_1}{y_n}\prec_\gamma \ssstree{z_1}{z_\ell}+\ssstree{y_1}{y_n}\succ_\beta\ssstree{z_1}{z_\ell}+\ssstree{y_1}{y_n}
\cdot\ssstree{z_1}{z_\ell}\Biggr)\\
&\hspace{3cm}(\text{by Eq.~(\mref{eq:tdpre})})\\
=& \ \bigv x_1,\ldots,x_{m-1},x_m;m+1;1,\ldots,1,\beta\gamma;\Biggl(|,\ldots,|,|\prec_{\beta\gamma}
\Biggl(\ssstree{y_1}{y_n}\prec_\gamma \ssstree{z_1}{z_\ell}+\ssstree{y_1}{y_n}\succ_\beta\ssstree{z_1}{z_\ell}+\ssstree{y_1}{y_n}\cdot\ssstree{z_1}{z_\ell}\Biggr)\\
& \ +|\succ_1\Biggl(\ssstree{y_1}{y_n}\prec_\gamma \ssstree{z_1}{z_\ell}+\ssstree{y_1}{y_n}\succ_\beta\ssstree{z_1}{z_\ell}+\ssstree{y_1}{y_n}\cdot\ssstree{z_1}{z_\ell}\Biggr)\\
& \ +|\cdot\Biggl(\ssstree{y_1}{y_n}\prec_\gamma \ssstree{z_1}{z_\ell}+\ssstree{y_1}{y_n}\succ_\beta\ssstree{z_1}{z_\ell}+\ssstree{y_1}{y_n}\cdot\ssstree{z_1}{z_\ell}\Biggr)\Biggr)\\
&\hspace{3cm}(\text{by Item~\mref{it:trida} in Definition~\mref{defn:trirec} and Eq.~(\mref{eq:con})})\\
=& \ \bigv x_1,\ldots,x_m;m+1;1,\ldots,1;(|,\ldots,|)\prec_{\beta\gamma}\Biggl(\ssstree{y_1}{y_n}\prec_\gamma \ssstree{z_1}{z_\ell}+\ssstree{y_1}{y_n}\succ_\beta\ssstree{z_1}{z_\ell}+\ssstree{y_1}{y_n}\cdot\ssstree{z_1}{z_\ell}\Biggr)\\
&\hspace{3cm}(\text{by Eq.~(\mref{eq:tdpre})})\\
=& \ \ssstree{x_1}{x_m}\prec_{\beta\gamma}\Biggl(\ssstree{y_1}{y_n}\prec_\gamma \ssstree{z_1}{z_\ell}+\ssstree{y_1}{y_n}\succ_\beta\ssstree{z_1}{z_\ell}+\ssstree{y_1}{y_n}\cdot\ssstree{z_1}{z_\ell}\Biggr)\\
=& \ T\prec_{\beta\gamma}(U\prec_\gamma W+U\succ_\beta W+U\cdot W).
\end{align*}}%
The proofs of Eqs.~(\mref{eq:tdf2})-(\mref{eq:tdf7}) are similar.
For the induction step of $\dep(T)+\dep(U)+\dep(W)\geq 4$, we have
\begin{align*}
&(T\prec_\beta U)\prec_\gamma W\\
=& \ \bigv x_1,\ldots,x_{m-1},x_m;m+1;\alpha_0,\ldots,\alpha_{m-1},
\alpha_m\beta;(T^{(0)},\ldots,T^{(m-1)},T^{(m)}\succ_{\alpha_m}U+T^{(m)}\prec_{\beta}U+T^{(m)}\cdot U)\prec_\gamma W\quad(\text{by Eq.~(\mref{eq:tdpre})})\\
=& \ \bigv x_1,\ldots,x_{m-1},x_m;m+1;\alpha_0,\ldots,\alpha_{m-1},\alpha_m\beta\gamma;
\Big(T^{(0)},\ldots,T^{(m-1)},(T^{(m)}
\succ_{\alpha_m}U+T^{(m)}\prec_{\beta}U
+T^{(m)}\cdot U)\prec_\gamma W\\
& \ +(T^{(m)}\succ_{\alpha_m}U+T^{(m)}\prec_{\beta}U
+T^{(m)}\cdot U) \succ_{\alpha_m\beta}W+(T^{(m)}\succ_{\alpha_m}U +T^{(m)}\prec_{\beta}U+T^{(m)}\cdot U)\cdot W\Big)\\
&\hspace{6cm}(\text{by Eq.~(\mref{eq:tdpre})})\\
=& \ \bigv x_1,\ldots,x_{m-1},x_m;m+1;\alpha_0,\ldots,\alpha_{m-1},\alpha_m\beta\gamma;
\Big(T^{(0)},\ldots,T^{(m-1)},(T^{(m)}\succ_{\alpha_m}U)\prec_\gamma W+(T^{(m)}\prec_\beta U)\prec_\gamma W+(T^{(m)}\cdot U)\prec_\gamma W\\
& \ +T^{(m)}\succ_{\alpha_m}(U\succ_\beta W)+(T^{(m)}\succ_{\alpha_m}U+T^{(m)}\prec_{\beta}U+T^{(m)}\cdot U)\cdot W\Big)\\
=& \ \bigv x_1,\ldots,x_{m-1},x_m;m+1;\alpha_0,\ldots,\alpha_{m-1},\alpha_m\beta\gamma;\Big(T^{(0)},
\ldots,T^{(m-1)},T^{(m)}\succ_{\alpha_m}(U\prec_\gamma W)+T^{(m)}\prec_{\beta\gamma}(U\prec_\gamma W+U\succ_\beta W+U\cdot W)\\
& \ +(T^{(m)}\cdot U)\prec_\gamma W+T^{(m)}\succ_{\alpha_m}(U\succ_\beta W)+(T^{(m)}\succ_{\alpha_m} U)\cdot W+(T^{(m)}\prec_\beta U)\cdot W+(T^{(m)}\cdot U)\cdot W\Big)\\
&\hspace{5cm}(\text{by the induction hypothesis})\\
=& \ \bigv x_1,\ldots,x_{m-1},x_m;m+1;\alpha_0,\ldots,\alpha_{m-1},\alpha_m\beta\gamma;
\Big(T^{(0)},\ldots,T^{(m-1)},T^{(m)}\prec_{\beta\gamma}(U\prec_\gamma W+U\succ_\beta W+U\cdot W)+T^{(m)}\succ_{\alpha_m}(U\prec_\gamma W+\\
& \ U\succ_\beta W+U\cdot W)+T^{(m)}\cdot(U\prec_\gamma W+U\succ_\beta W+U\cdot W)\Big)\\
=& \ \bigv x_1,\ldots,x_m;m+1;\alpha_0,\ldots,\alpha_m;(T^{(0)},\ldots,T^{(m)})\prec_{\beta\gamma}(U\prec_\gamma W+U\succ_\beta W+U\cdot W)\quad(\text{by Eq.~(\mref{eq:tdpre})})\\
=& \ T\prec_{\beta\gamma}(U\prec_\gamma W+U\succ_\beta W+U\cdot W).
\end{align*}
The proofs of Eqs.~(\mref{eq:tdf2})-(\mref{eq:tdf7}) are similar and are moved to Appendix.

{\noindent \bf Step 2:}
We show that $\DTF$ is generated by $\biggl\{\stree x\Bigm|x\in X\biggr\}$. We employ induction on $\dep(T)\geq 1$.
For the initial step $\dep(T)=1,$ we have $T=\ssstree{x_1}{x_m}=\stree{x_1}\cdot\stree{x_2}\cdot \ldots \cdot\stree{x_m}$ 
where $\stree{x_i}\in \biggl\{\stree x\Bigm|x\in X\biggr\}$ with $1\leq i\leq m$.
For any induction step $\dep(T)\geq 2$,
we apply a secondary induction on $\bre(T)\geq 2.$

For the initial step $\bre(T)=2$,
since $\dep(T)\geq 2$, we have $T\neq \stree x$ for any $x\in X$
and there are three cases to consider.

{\noindent \bf Case 1:} $T^{(0)}=|$ and $T^{(1)}\neq |$. Then $\alpha_0=1$ and
\begin{align*}
T=& \ \bigv x;2;\alpha_0,\alpha_1;(T^{(0)},T^{(1)})=\bigv x;2;1,\alpha_1;(|,T^{(1)})\\
=& \  \bigv x;2;1,\alpha_1;(|,|\prec_{\alpha_1}T^{(1)}+|\succ_1T^{(1)}+|\cdot T^{(1)})\quad(\text{by Item~\mref{it:trida} in Definition~\mref{defn:trirec} and Eq.~(\mref{eq:con})})\\
=& \ \bigv x;2;1,1;(|,|)\prec_{\alpha_1}T^{(1)}\quad(\text{by Eq.~(\mref{eq:tdpre})})\\
=& \ \stree x\prec_{\alpha_1}T^{(1)}.
\end{align*}
By the induction hypothesis, $T^{(1)}$ is generated by $\biggl\{\stree x\Bigm|x\in X\biggr\}$ and so $T$ is generated by $\biggl\{\stree x\Bigm|x\in X\biggr\}.$

{\noindent \bf Case 2:} \mlabel{subc:1.2} $T^{(0)}\neq |$ and $T^{(1)}=|$. This case is similar to Case 1.

{\noindent \bf Case 3:} $T^{(0)}\neq |$ and $T^{(1)}\neq|$. Then
\begin{align*}
T=& \ \bigv x;2;\alpha_0,\alpha_1;(T^{(0)},T^{(1)})\\
=& \ \bigv x;2;\alpha_0,\alpha_1;(T^{(0)},|\prec_{\alpha_1}T^{(1)}+|\succ_1T^{(1)}
+|\cdot T^{(1)})\quad(\text{by Item~\mref{it:trida} in Definition~\mref{defn:trirec} and Eq.~(\mref{eq:con})})\\
=& \ \bigv x;2;\alpha_0,1;(T^{(0)},|)\prec_{\alpha_1}T^{(1)}\quad(\text{by Eq.~(\mref{eq:tdpre})})\\
= & \ \biggl(\bigv x;2;\alpha_0,1;(T^{(0)}\succ_{\alpha_0}|+T^{(0)}\prec_1|+T^{(0)}\cdot|
,|)\biggr)\prec_{\alpha_1}T^{(1)}\,(\text{by Item~\mref{it:trida} in Definition~\mref{defn:trirec} and Eq.~(\mref{eq:con})})\\
=& \ \biggl(T^{(0)}\succ_{\alpha_0}\bigv x;2;1,1;(|,|)\biggr)\prec_{\alpha_1}T^{(1)}\quad(\text{by Eq.~(\mref{eq:tdsuc})})\\
=& \ \biggl(T^{(0)}\succ_{\alpha_0}\stree x\biggr)\prec_{\alpha_1}T^{(1)}.
\end{align*}
By the induction hypothesis, $T^{(0)}$ and $T^{(1)}$ are generated by $\biggl\{\stree x\Bigm|x\in X\biggr\}$ and so $T$ is generated by $\biggl\{\stree x\Bigm|x\in X\biggr\}.$

Now let us consider the case when $\dep(T)=k+1\geq 2$ and $\bre(T)=m+1\geq 3.$ Then we can write $T=\bigv x_1,\ldots,x_{m};m+1;\alpha_0,\ldots,\alpha_m;(T^{(0)},\ldots,T^{(m)})$, there exist four cases to consider.

{\noindent \bf Case 4:} $T^{(0)}=|$ and $T^{(1)}\neq |$. Then $\alpha_0=1$ and
\begin{align*}
T=& \ \bigv x_1,x_2,\ldots,x_m;m+1;1,\alpha_1,\ldots,\alpha_m;(|,T^{(1)},\ldots,T^{(m)})\\
=& \ \biggl(\bigv x_1;2;1,\alpha_1;(|,T^{(1)})\biggr)\cdot \biggl(\bigv x_2;2;1,\alpha_2;(|,T^{(2)})\biggr)\cdot\ldots\cdot \biggl(\bigv x_m,;2;1,\alpha_m;(|,T^{(m)})\biggr)\\
& \ \quad(\text{by Item~\mref{it:trida} in Definition~\mref{defn:trirec} and Eq.~(\mref{eq:tdcdot})})\\
=& \ \biggl(\stree{x_1}\prec_{\alpha_1}T^{(1)}\biggr)\cdot \biggl(\stree{x_2}\prec_{\alpha_2}T^{(2)}\biggr)\cdot\ldots\cdot\biggl(\stree{x_m}\prec_{\alpha_m}T^{(m)}\biggr).\quad(\text{by Case 1})
\end{align*}
By the induction hypothesis, $T^{(i)}$ with $1\leq i\leq m$ are generated by $\biggl\{\stree x\Bigm|x\in X\biggr\}$ and so $T$ is generated by $\biggl\{\stree x\Bigm|x\in X\biggr\}.$

{\noindent \bf \bf Case 5:} $T^{(0)}\neq |$ and $T^{(1)}=|$. This case is similar to Case 4.

{\noindent \bf Case 6:} $T^{(0)}\neq |$ and $T^{(1)}\neq|$. Then
\begin{align*}
T=& \ \bigv x_1,x_2,\ldots,x_m;m+1;\alpha_0,\alpha_1,\ldots,\alpha_m;
(T^{(0)},T^{(1)},\ldots,T^{(m)})\\
=& \ \biggl(\bigv x_1;2;\alpha_0,\alpha_1;(T^{(0)},T^{(1)})\biggr)\cdot \biggl(\bigv x_2;2;1,\alpha_2;(|,T^{(2)})\biggr)\cdot\ldots\cdot \biggl(\bigv x_m,;2;1,\alpha_m;(|,T^{(m)})\biggr)\\
& \ \quad(\text{by Item~\mref{it:trida} in Definition~\mref{defn:trirec} and Eq.~(\mref{eq:tdcdot})})\\
=& \ \Biggl(\biggl(T^{(0)}\succ_{\alpha_0}\stree {x_1}\biggr)\prec_{\alpha_1}T^{(1)}\Biggr)\cdot \biggl(\stree{x_2}\prec_{\alpha_2}T^{(2)}\biggr)\cdot\ldots\cdot\biggl(\stree {x_m}\prec_{\alpha_m}T^{(m)}\biggr).
\quad(\text{by Case 1 and Case 3})
\end{align*}
By the induction hypothesis, $T^{(i)}$ with $1\leq i\leq m$ are generated by $\biggl\{\stree x\Bigm|x\in X\biggr\}$ and so $T$ is generated by $\biggl\{\stree x\Bigm|x\in X\biggr\}.$

{\noindent \bf Case 7:} $T^{(0)}=|$ and $T^{(1)}=|$. Then
\begin{align*}
T=& \ \bigv x_1,x_2,x_3,\ldots,x_m;m+1;1,1,\alpha_2,\ldots,\alpha_m;(|,|,T^{(2)},\ldots, T^{(m)})\\
=& \ \biggl(\bigv x_1;2;1,1;(|,|)\biggr)\cdot\biggl(\bigv x_2;2;1,\alpha_2;(|,T^{(2)})\biggr)\cdot\ldots\cdot \biggl(\bigv x_m;2;1,\alpha_m;(|,T^{(m)})\biggr)\\
& \ \quad(\text{by Item~\mref{it:trida} in Definition~\mref{defn:trirec} and Eq.~(\mref{eq:tdcdot})})\\
=& \ \stree {x_1}\cdot \biggl(\stree {x_2}\prec_{\alpha_2}T^{(2)}\biggr)\cdot\ldots\cdot\biggl(\stree {x_m}\prec_{\alpha_m}T^{(m)}\biggr).\quad(\text{by Case 1})
\end{align*}
By the induction hypothesis, $T^{(i)}$ with $1\leq i\leq m$ are generated by $\biggl\{\stree x\Bigm|x\in X\biggr\}$ and so is $T$.
This completes the proof.
\end{proof}

\noindent The concept of free tridendriform family algebra is given as usual.
\begin{defn}
 Let $X$ be a set and let $\Omega$ be a semigroup.
 A free tridendriform family algebra on $X$ is a tridendriform family algebra $(T, \{\prec_{\omega},\succ_{\omega}\mid \omega\in \Omega\},  \cdot\, )$ together with the map $j: X\rightarrow T$ that satisfies the following universal property:
 for any tridendriform family algebra $(T', \{\prec'_{\omega},\succ'_{\omega}\mid \omega\in \Omega\},  \cdot'\,)$
and map $f: X\rightarrow T',$ there is a unique tridendriform family algebra morphism $\bar{f}:
T\rightarrow T'$ such that $f=\bar{f}\circ j$. The free tridendriform family algebra is unique up to isomorphism.
\end{defn}
Let $j: X \ra \DTF$ be the map defined by $j(x)=\stree x$ for $x\in X$. Now we arrive at our main result in this section.
\begin{theorem}
Let $X$ be a set and let $\Omega$ be a semigroup. The tridendriform family algebra $(\DTF,\{\prec_\omega,\succ_\omega\mid\omega\in\Omega\},\cdot)$, together with the map $j$, is the free tridendriform family algebra on $X$.
\mlabel{thm:free2}
\end{theorem}

\begin{proof}
By Proposition~\mref{prop:trid}, we are left to prove $(\DTF,\{\prec_\omega,\succ_\omega\mid\omega\in\Omega\},\cdot)$ satisfies the universal property. For this, let $(T', \{\prec'_{\omega},\succ'_{\omega}\mid \omega\in \Omega\},  \cdot'\,)$ be a tridendriform family algebra. First of all, if there exists a tridendriform family algebra morphism $\bar{f}: \DTF\to T'$ extending $f:X\to T'$ in the sense that $\bar f(\stree{x})=f(x)$ for any $x\in X$, then such an $\bar f$ is unique, due to the fact that the collection of trees $\left(\stree{x}\right)_{x\in X}$ generate the tridendriform family algebra $\DTF$.
Now we define a linear map
$$\bar{f}: \DTF\rightarrow T', \quad T \mapsto \bar{f}(T)$$
by induction on $\dep(T)\geq 1$ and prove that it is a tridendriform family algebra morphism. We use the unique decomposition
\begin{align}\label{central-product-decomp}
T=&\bigv x_1,\ldots,x_m;m+1;\alpha_0,\ldots,\alpha_m;(T^{(0)},\ldots,T^{(m)})\nonumber\\
=&\bigv x_1;2;\alpha_0,\alpha_1;(T^{(0)},T^{(1)})\cdot\bigv x_2;2;1,\alpha_2;(|,T^{(2)})\cdot\ldots\cdot\bigv x_m;2;1,\alpha_m;(|,T^{(m)}),
\end{align}
together with:
\begin{align*}
\bigv x;2;\alpha_0,\alpha_1;(V,W)&=\stree{x} \hbox{ if } V=W=|,\\
&=V\succ_{\alpha_0} \stree{x}\hbox { if } V\neq | \hbox{ and } W=|,\\
&=\stree{x}\prec_{\alpha_1} W\hbox{ if } V=| \hbox{ and } W\neq |,\\
&=V\succ_{\alpha_0}\stree{x}\prec_{\alpha_1} W\hbox{ if } V\neq | \hbox{ and } W\neq |,
\end{align*}
which uniquely defines $\bar{f}(T)$ in terms of $\bar{f}(T^{(0)}),\ldots,\bar{f}(T^{(m)})$ and the operations $(\{\prec'_\omega,\succ'_\omega\,\mid \omega\in\Omega\},\cdot')$.
\ignore
{we apply the secondary induction on $\bre(T)\geq 2$. For the initial step $\bre(T)=2$, since $\dep(T)\geq 2$, we have $T\neq \stree x$ for any $x\in X$ and define
\begin{align}
\bar{f}(T):=& \ \bar{f}\biggl(\bigv x;2;\alpha_0,\alpha_1;(T^{(0)},T^{(1)})\biggr)\nonumber\\
:=& \ \left\{
\begin{array}{ll}
f(x)\prec'_{\alpha_1}\bar{f}(T^{(1)}), & \ \text{ if }\,T^{(0)}=|\neq T^{(1)};\\
\bar{f}(T^{(0)})\succ'_{\alpha_0} f(x), & \ \text{ if } T^{(0)}\neq |=T^{(1)};\\
\bigl(\bar{f}(T^{(0)})\succ'_{\alpha_0}  f(x)\bigr)\prec'_{\alpha_1}\bar{f}(T^{(1)}), & \ \text{ if }\,T^{(0)}\neq|\neq T^{(1)}.\mlabel{eq:tde3}
\end{array}
\right .
\end{align}
}
We are going to prove that $\bar{f}$ is a morphism of tridendriform family algebras:
$$\bar{f}(T\prec_\omega U)=\bar{f}(T)\prec'_\omega \bar{f}(U),\,\bar{f}(T\succ_\omega U)=\bar{f}(T)\succ'_\omega\bar{f}(U)\,\text{ and }\,\bar{f}(T\cdot U)=\bar{f}(T)\cdot'\bar{f}(U),$$
in which we only prove the first equation, as the proofs of the two other are similar. Starting from
$$T=\bigv x_1,\ldots,x_m;m+1;\alpha_0,\ldots,\alpha_m;(T^{(0)},\ldots,T^{(m)}), \hskip 8mm U=\bigv y_1,\ldots,y_n;n+1;\beta_0,\ldots,\beta_n;(U^{(0)},\ldots,U^{(n)})$$
we compute:
\begin{align}\label{TDF-morphism}
&\bar f(T\prec_\omega U)\nonumber\\
&=\bar f\left(\bigv x_1,\ldots, x_m;m+1;\alpha_0,\ldots,\alpha_{m-1},\alpha_m\omega;(T^{(0)},\ldots, T^{(m-1)}, T^{(m)}\succ_{\alpha_m} U+T^{(m)}\prec_\omega U+T^{(m)}\cdot U)\right)\nonumber\\
&=\bar f\left(\bigv x_1;2;\alpha_0,\alpha_1;(T^{(0)}, T^{(1)})\cdot\bigv x_2;2;1,\alpha_2;(|,T^{(2)})\cdot\ldots\cdot
\bigv x_m;2;1,\alpha_m\omega;(|,T^{(m)}\succ_{\alpha_m}U+T^{(m)}\prec_{\omega}U+T^{(m)}\cdot U)\right)\nonumber\\
&\hskip 50mm\hbox{ (from }\eqref{central-product-decomp})\nonumber\\
&=\bar f \left(\bigv x_1;2;\alpha_0,\alpha_1;(T^{(0)}, T^{(1)})\right)\cdot'\ldots\cdot'
\bar f\left(\stree{x_m}\prec_{\alpha_m\omega}(T^{(m)}\succ_{\alpha_m}U+T^{(m)}\prec_{\omega}U+T^{(m)}\cdot U)\right)\nonumber\\
&\hskip 50mm\hbox{ (from the definition of }\bar f)\nonumber\\
&=\bar f \left(\bigv x_1;2;\alpha_0,\alpha_1;(T^{(0)}, T^{(1)})\right)\cdot'\ldots\cdot'
\bar f\left(\Bigl(\stree{x_m}\prec_{\alpha_m} T^{(m)}\Bigr)\prec_{\omega} U\right)\hbox{ (from Axiom }\eqref{eq:tdf1}).
\end{align}
Now we proceed by induction on $\dep(T)+\dep(U)$, together with a secondary induction on $\bre(T)$. If $\dep(T)+\dep(U)=2$ (main initial step), then $T^{(0)}=\cdots=T^{(m)}=|$, and Equation~\eqref{TDF-morphism} boils down to
\begin{align*}
\bar f(T\prec_\omega U)&=\bar f\Bigl(\stree{x_1}\bigr)\cdot'\ldots\cdot'\bar f\Bigl(\stree{x_{m-1}}\Bigr)\cdot'\bar f\left(\stree{x_m}\prec_\omega U\right)\\
&=\bar f\Bigl(\stree{x_1}\Bigr)\cdot'\ldots\cdot'\bar f\Bigl(\stree{x_{m-1}}\Bigr)\cdot'\left(\bar f\Bigl(\stree{x_m}\Bigr)\prec'_\omega \bar f(U)\right)\hbox{ (from the definition of }\bar f)\\
&=\left(\bar f\Bigl(\stree{x_1}\Bigr)\cdot'\ldots\cdot'\bar f\Bigl(\stree{x_{m-1}}\Bigr)\cdot'\bar f\Bigl(\stree{x_m}\Bigr)\right)\prec'_\omega \bar f(U)\hbox{ (from an iteration of Axiom } \eqref{eq:tdf6})\\
&=\bar f (T)\prec'_\omega\bar f(U).
\end{align*}
If $\dep(T)+\dep(U)\ge 3$ and $\bre(T)=2$ (secondary initial step), then we have
\begin{align*}
\bar f(T\prec_\omega U)&=\bar f\left(\bigv x_1;2;\alpha_0,\alpha_1\omega;\bigl(T^{0)}, T^{(1)}\succ_{\alpha_1}U+T^{(1)}\prec_\omega U+T^{(1)}\cdot U\bigr)\right)\\
&=\bar f\left(T^{(0)}\succ_{\alpha_0}\stree{x_1}\prec_{\alpha_1\omega}\bigl(T^{(1)}\succ_{\alpha_1}U+T^{(1)}\prec_\omega U+T^{(1)}\cdot U\bigr)\right)\\
&=\bar f(T^{(0)})\succ'_{\alpha_0}\bar f\Bigl(\stree{x_1}\Bigr)\prec'_{\alpha_1\omega}\bar f\bigl(T^{(1)}\succ_{\alpha_1}U+T^{(1)}\prec_\omega U+T^{(1)}\cdot U\bigr) \hbox{ (by definition of }\bar f)\\
&=\bar f(T^{(0)})\succ'_{\alpha_0}\bar f\Bigl(\stree{x_1}\Bigr)\prec'_{\alpha_1\omega}\Bigl(\bar f(T^{(1)})\succ'_{\alpha_1}\bar f(U)+\bar f(T^{(1)})\prec'_\omega \bar f(U)+\bar f(T^{(1)})\cdot' \bar f(U)\Bigr)\\
&\hskip 50mm\hbox{ (by the main induction hypothesis)}\\
&=\bar f(T^{(0)})\succ'_{\alpha_0}\left(\left(\bar f\Bigl(\stree{x_1}\Bigr)\prec'_{\alpha_1}\bar f(T^{(1)})\right)\prec'_\omega\bar f(U)\right)\\
&=\left(\bar f(T^{(0)})\succ'_{\alpha_0}\left(\bar f\Bigl(\stree{x_1}\Bigr)\prec'_{\alpha_1}\bar f(T^{(1)})\right)\right)\prec'_\omega\bar f(U)\\
&=\bar f(T)\prec'_\omega \bar f(U)\hbox{ (again by definition of }\bar f).
\end{align*}
Finally, if $\dep(T)+\dep(U)\ge 3$ and $\bre(T)\ge 3$, we can push further Compilation \eqref{TDF-morphism} by using the secondary initial step for $\stree{x_m}\prec_{\alpha_m}T^{(m)}$, which is of breadth two:
\begin{align*}
\bar f(T\prec_\omega U)&=\bar f \Bigl(\bigv x_1;2;\alpha_0,\alpha_1;(T^{(0)}, T^{(1)})\Bigr)\cdot'\ldots\cdot'
\left(\bar f\Bigl(\stree{x_m}\prec_{\alpha_m} T^{(m)}\Bigr)\prec'_{\omega} \bar f(U)\right)\\
&=\left(\bar f \Bigl(\bigv x_1;2;\alpha_0,\alpha_1;(T^{(0)}, T^{(1)})\Bigr)\cdot'\ldots\cdot'
\bar f\Bigl(\stree{x_m}\prec_{\alpha_m} T^{(m)}\Bigr)\right)\prec'_{\omega} \bar f(U)\\
&\hskip 50mm\hbox{ (by iterating Axiom }\eqref{eq:tdf6})\\
&=\bar f\left(\Bigl(\bigv x_1;2;\alpha_0,\alpha_1;(T^{(0)}, T^{(1)})\Bigr)\cdot\ldots\cdot
\Bigl(\stree{x_m}\prec_{\alpha_m} T^{(m)}\Bigr)\right)\prec'_{\omega} \bar f(U)\\
&\hskip 50mm\hbox{ (by definition of }\bar f)\\
&=\bar f(T)\prec'_\omega\bar f(U).
\end{align*}
The verification of the two identities $\bar{f}(T\succ_\omega U)=\bar{f}(T)\succ'_\omega\bar{f}(U)\,\text{ and }\,\bar{f}(T\cdot U)=\bar{f}(T)\cdot'\bar{f}(U)$ is similar and left to the reader.

\end{proof}

\smallskip

\noindent {\bf Acknowledgments}:
This work was supported by the National Natural Science Foundation
of China (Grant No.\@ 11771191 and 11861051). 

\section{Appendix}
The remaining proofs of Step $1$ in Proposition~\mref{prop:trid}.
{\small{
\begin{align*}
&(T\succ_\alpha U)\prec_\gamma W\\
=& \ \bigv y_1,y_2,\ldots,y_n;n+1;\alpha\beta_0,\beta_1,\ldots,\beta_n;\Big(T\succ_\alpha U^{(0)}+T\prec_{\beta_0} U^{(0)}+T\cdot U^{(0)}),U^{(1)},\ldots,U^{(n)}\Big)\prec_\gamma W\\
=& \ \bigv y_1,y_2,\ldots,y_{n-1},y_n;n+1;\alpha\beta_0,\beta_1,\ldots,\beta_{n-1},\beta_n\gamma;
\Big(T\succ_\alpha U^{(0)}+T\prec_{\beta_0} U^{(0)}+T\cdot U^{(0)},U^{(1)},\ldots,U^{(n-1)},\\
& \ U^{(n)}\succ_{\beta_n}W+U^{(n)}\prec_{\gamma}W+U^{(n)}\cdot W\Big)\\
=& \ T\succ_\alpha\Big(\bigv y_1,\ldots,y_{n-1},y_n;n+1;\beta_0,\ldots,\beta_{n-1},\beta_n\gamma;(U^{(0)},
\ldots,U^{(n-1)},U^{(n)}\succ_{\beta_n}W+U^{(n)}\prec_\gamma W+U^{(n)}\cdot W\Big)\\
=& \ T\succ_\alpha(U\prec_\gamma W).\\
\setlength{\baselineskip}{20pt}\\
 &T\succ_\alpha(U\succ_\beta W)\\
=& \ T\succ_\alpha\Big(\bigv z_1,z_2,\ldots,z_\ell;\ell+1;\beta\gamma_0,\gamma_1,\ldots,\gamma_\ell;(U\succ_\beta W^{(0)}+U\prec_{\gamma_0}W^{(0)}+U\cdot W^{(0)}, W^{(1)},\ldots,W^{(l)})\Bigr)\\
=& \ \bigv z_1,z_2,\ldots,z_\ell;\ell+1;\alpha\beta\gamma_0,\gamma_1,\ldots,\gamma_\ell;
\Big(T\succ_\alpha(U\succ_\beta W^{(0)}+U\prec_{\gamma_0}W^{(0)}+U\cdot W^{(0)})+T\prec_{\beta\gamma_0}(U\succ_\beta W^{(0)}
 +U\prec_{\gamma_0}W^{(0)}\\
& \  +U\cdot W^{(0)})+T\cdot(U\succ_\beta W^{(0)}+U\prec_{\gamma_0}W^{(0)}+U\cdot W^{(0)}),W^{(1)},\ldots,W^{(\ell)}\Big)\\
=& \ \bigv z_1,z_2,\ldots,z_\ell;\ell+1;\alpha\beta\gamma_0,\gamma_1,\ldots,\gamma_\ell;\Big(T\succ_\alpha(U\succ_\beta W^{(0)})+T\succ_\alpha(U\prec_{\gamma_0}W^{(0)})+T\succ_\alpha(U\cdot W^{(0)})+(T\prec_\beta U)\prec_{\gamma_0}W^{(0)}\\
& \ +T\cdot(U\succ_\beta W^{(0)})+T\cdot(U\prec_{\gamma_0}W^{(0)})+T\cdot(U\cdot W^{(0)}),W^{(1)},\ldots,W^{(\ell)}\Big)\\
=& \ \bigv z_1,z_2,\ldots,z_\ell;\ell+1;\alpha\beta\gamma_0,\gamma_1,\ldots,\gamma_\ell;\Big((T \succ_\alpha U+T\prec_\beta U+T\cdot U)\succ_{\alpha\beta}W^{(0)}+(T\succ_\alpha U+T\prec_\beta U+T\cdot U)\prec_{\gamma_0}W^{(0)}\\
& \ +(T\succ_\alpha U+T\prec_\beta U+T\cdot U)\cdot W^{(0)},W^{(1)},\ldots,W^{(\ell)}\Big)\\
=& \ (T\succ_\alpha U+T\prec_\beta U+T\cdot U)\succ_{\alpha\beta}(\bigv z_1,\ldots,z_\ell;\ell+1;\gamma_0,\ldots,\gamma_\ell;(W^{(0)},\ldots,W^{l})\\
=& \ (T\succ_\alpha U+T\prec_\beta U+T\cdot U)\succ_{\alpha\beta}W.\\
%
\setlength{\baselineskip}{20pt}\\
 &(T\succ_\alpha U)\cdot W\\
=& \ \bigv y_1,y_2,\ldots,y_n;n+1;\alpha\beta_0,\beta_1,\ldots,\beta_n;\Big(T\succ_\alpha U^{(0)}+T\prec_{\beta_0} U^{(0)}+T\cdot U^{(0)},U^{(1)},\ldots,U^{(n)}\Big)\cdot\bigv z_1,\ldots,z_\ell;\ell+1;\gamma_0,\ldots,\gamma_\ell;(W^{(0)},\ldots,W^{(l)})\\
=& \ \bigv y_1,y_2,\ldots,y_n,z_1,z_2,\ldots z_\ell;n+1+\ell;\alpha\beta_0,\beta_1,\ldots,\beta_{n-1},\beta_n\gamma_0,
\gamma_1,\ldots,\gamma_\ell;\Big(T\succ_\alpha U^{(0)}+T\prec_{\beta_0} U^{(0)}+T\cdot U^{(0)},U^{(1)},\ldots,U^{(n-1)},\\
& \ U^{(n)}\succ_{\beta_n}W^{(0)}
+U^{(n)}\prec_{\gamma_0}W^{(0)}+U^{(n)}\cdot W^{(0)},
W^{(1)},\ldots,W^{(l)}\Bigr)\\
=& \ T\succ_\alpha\bigv y_1,\ldots,y_{n-1},y_n,z_1,\ldots,z_\ell;n+1+\ell;\beta_0,\ldots,
\beta_{n-1},\beta_n\gamma_0,z_1,\ldots,z_\ell;\Big(U^{(0)},\ldots,U^{(n-1)},
U^{(n)}\succ_{\beta_n}W^{(0)}
+U^{(n)}\prec_{\gamma_0}W^{(0)}+U^{(n)}\cdot W^{(0)},\\
& \ W^{(1)},\ldots,W^{(l)}\Big)\\
=& \ T\succ_\alpha(U\cdot W).\\
\setlength{\baselineskip}{20pt}\\
 &(T\prec_\beta U)\cdot W\\
=& \ \bigv x_1,\ldots,x_{m-1},x_m;m+1;\alpha_0,\ldots,\alpha_{m-1},\alpha_m\beta;\Big(T^{(0)},
\ldots,T^{(m-1)},T^{(m)}\succ_{\alpha_m}U+T^{(m)}\prec_\beta U+T^{(m)}\cdot U\Big)\\
& \ \cdot\bigv z_1,\ldots,z_\ell;\ell+1;\gamma_0,\ldots,\gamma_\ell;(W^{(0)},\ldots,W^{(l)})\\
=& \ \bigv x_1,\ldots,x_{m-1},x_m,z_1,\ldots,z_\ell;m+1+\ell;\alpha_0,\ldots,\alpha_{m-1},
\alpha_m\beta\gamma_0,\gamma_1,\ldots,\gamma_\ell;\Big(T^{(0)},\ldots,T^{(m-1)},
(T^{(m)}\succ_{\alpha_m}U+T^{(m)}\prec_\beta U+T^{(m)}\cdot U)\succ_{\alpha_m\beta}W^{(0)}\\
& \ +(T^{(m)}\succ_{\alpha_m}U+T^{(m)}\prec_\beta U+T^{(m)}\cdot U)\prec_{\gamma_0}W^{(0)}+(T^{(m)}\succ_{\alpha_m}U+T^{(m)}\prec_\beta U\\
& \ +T^{(m)}\cdot U)\cdot W^{(0)},W^{(1)},\ldots,W^{(\ell)}\Big)\\
=& \ \bigv x_1,\ldots,x_{m-1},x_m,z_1,\ldots,z_\ell;m+1+\ell;
\alpha_0,\ldots,\alpha_{m-1},\alpha_m\beta\gamma_0,\gamma_1,\ldots,\gamma_\ell;\Big(T^{(0)},
\ldots,T^{(m-1)},T^{(m)}\succ_{\alpha_m}(U\succ_\beta W^{(0)}+U\prec_{\gamma_0}W^{(0)}+U\cdot W^{(0)})\\
& \ +T^{(m)}\prec_{\beta\gamma_0}(U\succ_\beta W^{(0)}+U\prec_{\gamma_0}W^{(0)}+U\cdot W^{(0)})+T^{(m)}\cdot(U\succ_\beta W^{(0)}+U\prec_{\gamma_0}W^{(0)}
+U\cdot W^{(0)}),\\
& \  W^{(1)},\ldots,W^{(\ell)}\Big)\\
=& \ \bigv x_1,\ldots,x_m;m+1;\alpha_0,\ldots,\alpha_m;(T^{(0)},\ldots,T^{(m)})\cdot \bigv z_1,z_2,\ldots,z_\ell;\ell+1;\beta\gamma_0,\gamma_1,\ldots,\gamma_\ell;\Big(U\succ_\beta W^{(0)}+U\prec_{\gamma_0}W^{(0)}+U\cdot W^{(0)},
W^{(1)},\ldots,W^{(l)}\Big)\\
=& \ T\cdot\bigv z_1,z_2,\ldots,z_\ell;\ell+1;\beta\gamma_0,\gamma_1,\ldots,\gamma_\ell;\Big(U\succ_\beta W^{(0)}+U\prec_{\gamma_0}W^{(0)}+U\cdot W^{(0)},
W^{(1)},\ldots,W^{(l)}\Big)\\
=& \ T\cdot(U\succ_\beta W).\\
\setlength{\baselineskip}{20pt}\\
&(T\cdot U)\prec_\gamma W\\
=& \ \bigv x_1,\ldots,x_m,y_1,y_2,\ldots,y_n;m+1+n;\alpha_0,\ldots,\alpha_{m-1},\alpha_m\beta_0,
\beta_1,\ldots,\beta_n;\Big(T^{(0)},\ldots,T^{(m-1)},T^{(m)}\succ_{\alpha_m} U^{(0)}+T^{(m)}\prec_{\beta_0} U^{(0)}+T^{(m)}\cdot U^{(0)},U^{(1)},\ldots,U^{(n)}\Big)\\
=& \ \bigv x_1,\ldots,x_{m-1},x_m,y_1,\ldots,y_{n-1},y_n;m+1+n;\alpha_0,\ldots,
\alpha_{m-1},\alpha_m\beta_0,\beta_1,\ldots,\beta_{n-1},\beta_n\gamma;\Big(T^{(0)},
\ldots,T^{(m-1)},T^{(m)}\succ_{\alpha_m} U^{(0)}+T^{(m)}\prec_{\beta_0} U^{(0)}+T^{(m)}\cdot U^{(0)},\\
& \ U^{(1)},\ldots,U^{(n-1)},
 U^{(n)}\succ_{\beta_n}W+U^{(n)}\prec_\gamma W+U^{(n)}\cdot W\Big)\\
=& \ \bigv x_1,\ldots,x_m;m+1;\alpha_0,\ldots,\alpha_m;(T^{(0)},\ldots,T^{(m)})\cdot \bigv y_1,\ldots,y_{n-1},y_n;n+1;\beta_0,\ldots,\beta_{n-1},\beta_n\gamma;
\Big(U^{(0)},\ldots,U^{(n-1)},U^{(n)}\succ_{\beta_n}W+U^{(n)}\prec_\gamma W+U^{(n)}\cdot W\Big)\\
=& \ T\cdot(U\prec_\gamma W).\\[20pt]
%
 &(T\cdot U)\cdot W\\
=& \ \bigv x_1,\ldots,x_{m-1},x_m,y_1,y_2,\ldots,y_n;m+1;\alpha_0,\ldots,
\alpha_{m-1},\alpha_m\beta_0,\beta_1,\ldots,\beta_n;\Big(T^{(0)},\ldots,T^{(m-1)},T^{(m)}\succ_{\alpha_m}U^{(0)}+T^{(m)}
\prec_{\beta_0}U^{(0)}
+T^{(m)}\cdot U^{(0)},
 U^{(0)},\ldots,U^{(n)}\Big)\\
& \  \cdot\bigv z_1,\ldots,z_\ell;\ell+1;\gamma_0,\ldots,\gamma_\ell;(W^{(0)},\ldots,W^{(l)})\\
=& \ \bigv x_1,\ldots,x_{m-1},x_m,y_1,\ldots,y_{n-1},y_n,z_1,\ldots,z_\ell;
m+1+n+\ell;\alpha_0,\ldots,\alpha_{m-1},\alpha_m\beta_0,\beta_1,\ldots,\beta_{n-1},
\beta_n\gamma_0,\gamma_1,\ldots,\gamma_\ell;\Big(T^{(0)},\ldots,T^{(m-1)},
T^{(m)}\succ_{\alpha_m}U^{(0)}+T^{(m)}
\prec_{\beta_0}U^{(0)}
+T^{(m)}\cdot U^{(0)},\\
& \ U^{(1)},\ldots,U^{(n-1)},
U^{(n)}\succ_{\beta_n}W^{(0)}+U^{(n)}\prec_{\gamma_0}W^{(0)}
+U^{(n)}\cdot W^{(0)}, W^{(1)},\ldots,W^{(l)}\Big)\\
=& \ \bigv x_1,\ldots,x_m;m+1;\alpha_0,\ldots,\alpha_m;(T^{(0)},\ldots,T^{(m)})\cdot\bigv y_1,\ldots,y_{n-1},y_n,z_1,\ldots,z_\ell;n+1;\beta_0,\ldots,
\beta_{n-1},\beta_n\gamma_0,\gamma_1,\ldots,\gamma_\ell;\Big(U^{(0)},\ldots,U^{(n-1)},
U^{(n)}\succ_{\beta_n}W^{(0)}\\
& \ +U^{(n)}\prec_{\gamma_0}W^{(0)}+U^{(n)}\cdot W^{(0)},W^{(1)},\ldots,W^{(\ell)}\Big)\\
=& \ T\cdot(U\cdot W).
\end{align*}}}

\end{document}